\nonstopmode

\documentclass{amsart}
\usepackage[hyperindex=true,pagebackref=true,bookmarks=true,bookmarksnumbered=true]{hyperref}

\usepackage{amsmath, amssymb, amsthm, amscd}
\usepackage{mathrsfs}
\usepackage{tensor}
\usepackage{empheq}

\newcommand*\widefbox[1]{\fbox{\hspace{1em}#1\hspace{1em}}}

\usepackage{verbatim}
\usepackage[all]{xy}
\newdir{ >}{{}*!/-10pt/@{>}}
\newdir^{ (}{{}*!/-5pt/@^{(}}
\newdir_{ (}{{}*!/-5pt/@_{(}}

\def\undersetbrace#1\to#2{\underbrace{#2}_{#1}}                                                          
\def\oversetbrace#1\to#2{\overbrace{#2}^{#1}}
\def\AMSunderset#1\to#2{\underset{#1}{#2}}
\def\AMSoverset#1\to#2{\overset{#1}{#2}}

\def\East#1#2{\overset{#1}{\longrightarrow}}

\swapnumbers
\newtheorem{proposition}[subsection]{Proposition}
\newtheorem*{proposition*}{Proposition}
\newtheorem{theorem}[subsection]{Theorem}
\newtheorem*{theorem*}{Theorem}
\newtheorem{lemma}[subsection]{Lemma}
\newtheorem*{lemma*}{Lemma}

\newtheorem*{corollary*}{Corollary}
\theoremstyle{definition}
\newtheorem*{definition*}{Definition}

\newtheorem*{remark*}{Remark}
\newtheorem{remark}[subsection]{Remark}
\newtheorem*{problem*}{Problem}

\newtheorem*{example*}{Example}

\newenvironment{demo}[1]{{\textit{#1.}}}{\par\smallskip}
\numberwithin{equation}{subsection}

\def\ign#1{}             
\def\o{\,\circ\,}

\def\al{\alpha}
\def\be{\beta}

\def\de{\delta}

\def\la{\lambda}
\def\rh{\rho}
\def\si{\sigma}
\def\ta{\tau}
\def\ph{\varphi}

\def\ps{\psi}
\def\om{\omega}

\def\Si{\Sigma}

\def\x{\times}
\def\p{\partial}
\let\on=\operatorname

\def\AMSonly#1{}

\def\Id{\on{Id}}
\def\R{\mathbb R}

\def\Diff{{\on{Diff}}}

\newcommand{\nmb}[2]{\ifx!#1{\ref{nmb:#2}}%
\else\if.#1{\label{nmb:#2}}%
\else\if0#1{\label{nmb:#2}}%
\else{{#2}}%
\fi\fi\fi}

\newcommand{\sr}[1]%
{\ifmmode{}^\dagger\else${}^\dagger$\fi\ifvmode
\vbox to 0pt{\vss
 \hbox to 0pt{\hskip\hsize\hskip1em
 \vbox{\hsize3cm\raggedright\pretolerance10000
 \noindent #1\hfill}\hss}\vss}\else
 \vadjust{\vbox to0pt{\vss%
 \hbox to 0pt{\hskip\hsize\hskip1em%
 \vbox{\hsize3cm\raggedright\pretolerance10000%
 \noindent #1\hfill}\hss}\vss}}\fi%
}

\def\C{\mathbb{C}}

\def\I{\mathbb{I}}

\def\N{\mathbb{N}}

\def\R{\mathbb{R}}

\def\cA{\mathcal{A}}
\def\cB{\mathcal{B}}

\def\cD{\mathcal{D}}

\def\cF{\mathcal{F}}

\def\cL{\mathcal{L}}

\def\cS{\mathcal{S}}

\def\cV{\mathcal{V}}

\def\sB{\mathscr{B}}

\def\sF{\mathscr{F}}

\def\sK{\mathscr{K}}

\def\sR{\mathscr{R}}
\def\sS{\mathscr{S}}

\def\fS{\mathfrak{S}}

\def\RR{\mathbb R}

\def\subs{\subseteq}
\def\A{\;\forall}
\def\E{\;\exists}
\def\oo{\infty}

\def\d{\partial}
\def\ev{\on{ev}}

\def\rM{\{M\}}
\def\bM{(M)}

\def\cBb{\cB_b}

\def\cAb{\cA_b}

\def\BM{\cB^{[M]}}
\def\BMb{\cB^{[M]}_b}
\def\BrM{\cB^{\rM}}
\def\BbM{\cB^{\bM}}

\def\DM{\cD^{[M]}}
\def\DrM{\cD^{\rM}}
\def\DbM{\cD^{\bM}}

\def\CM{C^{[M]}}
\def\CrM{C^{\rM}}
\def\CbM{C^{\bM}}

\def\cSb{\cS_b}

\def\SLM{\tensor{\cS}{}_{[L]}^{[M]}}
\def\SrLM{\tensor{\cS}{}_{\{L\}}^{\rM}}
\def\SbLM{\tensor{\cS}{}_{(L)}^{\bM}}

\def\SbLMb{\tensor{\cS}{}_{(L), b}^{\bM}}
\def\SrLMb{\tensor{\cS}{}_{\{L\}, b}^{\rM}}

\def\SrML{\tensor{\cS}{}_{\{M\}}^{\{L\}}}

\def\SLMone{\tensor{\cS}{}_{[L],\ge1}^{[M]}}
\def\SbLMone{\tensor{\cS}{}_{(L),\ge1}^{(M)}}
\def\SrLMone{\tensor{\cS}{}_{\{L\},\ge1}^{\{M\}}}

\def\Hoo{H^{\infty}}

\def\Sp{W^{\infty,p}}

\def\SpM{W^{[M],p}}
\def\SpbM{W^{\bM,p}}
\def\SprM{W^{\rM,p}}

\def\set{\Si}
\def\Lin{\cL}

\def\ind{\varinjlim}
\def\proj{\varprojlim}

\def\DiffA{\Diff\cA(\R^n)}
\def\DiffB{\Diff\cB(\R^n)}
\def\DiffD{\Diff\cD(\R^n)}
\def\DiffS{\Diff\cS(\R^n)}

\def\DiffSp{\Diff\Sp(\R^n)}

\def\SN{\mathscr{S\!\!N}}

\title[The exponential law for spaces of test functions] 
{The exponential law for spaces of test functions and diffeomorphism groups} 

  \author{Andreas Kriegl, Peter W. Michor, and Armin Rainer}
  
  \address{Andreas Kriegl: Fakult\"at f\"ur Mathematik, Universit\"at Wien, 
  Oskar-Morgenstern-Platz~1, A-1090 Wien, Austria}
  \email{andreas.kriegl@univie.ac.at}
  
  \address{Peter W. Michor: Fakult\"at f\"ur Mathematik, Universit\"at Wien, 
  Oskar-Morgenstern-Platz~1, A-1090 Wien, Austria}
  \email{peter.michor@univie.ac.at}
  
  \address{Armin Rainer: Fakult\"at f\"ur Mathematik, Universit\"at Wien, 
  Oskar-Morgenstern-Platz~1, A-1090 Wien, Austria}
  \email{armin.rainer@univie.ac.at}
  
  \thanks{AK was supported by FWF-Project P~23028-N13; 
  AR by FWF-Project P~26735-N25}
  \date{September 29, 2015}

\begin{document}

\begin{abstract}
  We prove the exponential law $\cA(E \times F, G) \cong \cA(E,\cA(F,G))$ (bornological isomorphism) 
  for the following classes $\cA$ of test functions:
  $\cB$ (globally bounded derivatives), $\Sp$ (globally $p$-integrable derivatives), $\cS$ (Schwartz space), 
  $\cD$ (compact support), $\BM$ (globally Denjoy--Carleman),
  $\SpM$ (Sobolev--Denjoy--Carleman),  
  $\SLM$ (Gelfand--Shilov), and $\DM$ (Denjoy--Carleman with compact support). 
  Here $E, F, G$ are convenient vector spaces which are finite dimensional in the cases of $\cD$, $\Sp$, 
  $\DM$, and  $\SpM$. Moreover,
  $M=(M_k)$ is a weakly log-convex weight sequence of moderate growth. 
  As application we give a new simple proof of the fact that 
  the groups of diffeomorphisms
  $\Diff \cB$, $\Diff \Sp$, $\Diff \cS$, and $\Diff\cD$ are $C^\infty$ Lie groups, and that
  $\Diff \BrM$, $\Diff \SprM$, $\Diff \SrLM$, and $\Diff\DrM$, for non-quasianalytic $M$, are $\CrM$ Lie groups,
  where $\Diff\cA = \{\Id +f : f \in \cA(\R^n,\R^n), \inf_{x \in \R^n} \det(\I_n+ df(x))>0\}$.   
  We also discuss stability under composition. 
\end{abstract}

\maketitle 

\section{Introduction}

In this paper we prove the bornological isomorphism  
\begin{equation} \label{eq:exp}
  \cA(E \times F, G) \cong \cA(E,\cA(F,G))
\end{equation}
for several classes $\cA$ of test functions. It is called \emph{exponential law}, since it takes the form 
$G^{E \times F} = (G^F)^E$ if one writes $\cA(X,Y) = Y^X$. 

The exponential law \eqref{eq:exp} is well-known in the categories of $C^\infty$, real analytic, and holomorphic functions; 
see \cite{KM97}.
In \cite{KMRc}, \cite{KMRq}, and \cite{KMRu} we established the exponential law \eqref{eq:exp} for local Denjoy--Carleman
classes $\CM$, provided that $M=(M_k)$ is weakly log-convex and 
has moderate growth. (The notation $\CM$ stands for the classes $\CrM$ of Roumieu type as well as for the classes $\CbM$ of Beurling type,
cf.\ Subsection~\ref{ssec:DC}.)
In all these cases the underlying spaces $E,F,G$ are so-called \emph{convenient} vector spaces, i.e.,
locally convex spaces that are Mackey complete.

We shall prove \eqref{eq:exp} for the following classes $\cA$ of test functions (see Sections \ref{sec:convenient} and \ref{sec:Wexp} for the 
precise definitions):
\begin{itemize}
   \item Smooth functions with globally bounded derivatives $\cB$ ($=\cD_{L^\infty}$ in \cite{Schwartz66})
   \item Smooth functions with $p$-integrable derivatives $\Sp$ ($=\cD_{L^p}$ in \cite{Schwartz66})
   \item Rapidly decreasing Schwartz functions $\cS$
	 \item Smooth functions with compact support $\cD$
   \item Global Denjoy--Carleman classes $\BM$
   \item Sobolev--Denjoy--Carleman classes $\SpM$
   \item Gelfand--Shilov classes $\SLM$
	 \item Denjoy--Carleman functions with compact support $\DM$   
 \end{itemize} 
For the sequence $L=(L_k)$ we just assume $L_k\ge 1$ for all $k$.

The underlying spaces are again convenient vector spaces, except for $\cD$, $\Sp$, $\DM$, and $\SpM$ 
when $E,F, G$ are assumed 
to be finite dimensional. The definition of the classes $\cB$, $\cS$, $\BM$, and $\SLM$ makes obvious sense between arbitrary Banach spaces.
By definition, a $C^\infty$-mapping $f : E \to F$ between general convenient vector spaces belongs to the class if the composite 
$\ell \o f \o i_B : E_B \to \R$ is in the class for each continuous linear functional $\ell : F \to \R$ and each closed 
absolutely convex bounded subset $B \subseteq E$, where $i_B : E_B \to E$ denotes the inclusions of the linear span $E_B$ of $B$ 
which equipped with the Minkowski functional is a Banach space. 

For finite dimensional parameter spaces we have the following continuous inclusions, where $1 \le p<q<\infty$;
for the inclusions marked by $*$ we assume that $M=(M_k)$ is derivation closed. 
    \[
      \xymatrix{
        \cD~ \ar@{{ >}->}[r] & \cS~ \ar@{{ >}->}[r] & \Sp~ \ar@{{ >}->}[r] & W^{\infty,q}~ \ar@{{ >}->}[r] 
        & \cB~ \ar@{{ >}->}[r] & C^\infty \\
        \DrM~ \ar@{{ >}->}[r] \ar@{{ >}->}[u] & \SrLM~ \ar@{{ >}->}[r] \ar@{{ >}->}[u] & \SprM~ \ar@{{ >}->}[r]^{*} \ar@{{ >}->}[u] 
        & W^{\rM,q}~ \ar@{{ >}->}[r]^{*} \ar@{{ >}->}[u] 
        & \BrM~ \ar@{{ >}->}[u] \ar@{{ >}->}[r] & \CrM~ \ar@{{ >}->}[u]\\
        \DbM~ \ar@{{ >}->}[r] \ar@{{ >}->}[u] &\SbLM~ \ar@{{ >}->}[r] \ar@{{ >}->}[u] & \SpbM~ \ar@{{ >}->}[r]^{*} \ar@{{ >}->}[u] 
        & W^{\bM,q}~ \ar@{{ >}->}[r]^{*} \ar@{{ >}->}[u] 
        & \BbM~ \ar@{{ >}->}[u] \ar@{{ >}->}[r] & \CbM~ \ar@{{ >}->}[u]
      }
    \]

We are grateful to a referee who pointed out that
\begin{equation} \label{eq:intro1}
 \cD(\R^\ell \times \R^m,\R^n) \cong \cD(\R^\ell,\cD(\R^m,\R^n)) 
\end{equation}
does not hold  topologically, contrary to, e.g., \cite[p.~415]{Treves67}. Namely, the right hand side is 
not barrelled since it contains a complemented subspace isomorphic to the projective tensor 
product of the space of rapidly decreasing sequences with the space of finite sequences which is 
complete but not barrelled, see \cite[Chap II, \S 4, no~1, proposition 14]{Grothendieck55}. But 
\eqref{eq:intro1} holds bornologically, i.e., in the category of convenient vector spaces, see 
Theorem \ref{thm:Dexp}.

Every continuous function with compact support in an infinite dimensional Banach space is identically zero.  
So it makes little sense to go beyond finite dimensional vector spaces in \eqref{eq:intro1}.
Note that, as $\DM \subseteq \SLM$, $\SLM$ is certainly non-trivial if $M=(M_k)$ is 
non-quasianalytic.

The paper is organized as follows.
Due to fundamental differences in the proofs for the classes defined by means of $L^\infty$-estimates on one hand 
and $L^p$-estimates on the other hand, we treat these cases separately.  
After collecting preliminaries on weight sequences in 
Section \ref{sec:prelim}, we devote the Sections \ref{sec:convenient} and \ref{sec:workingup}
to working up to the $\cB$, $\cS$, $\BM$, and $\SLM$
exponential law which is finally proved in Section \ref{sec:BSexp}: 
We introduce the respective classes of mappings between Banach spaces and extend 
them to convenient vector spaces in Section \ref{sec:convenient}. In Section \ref{sec:workingup}
we provide projective descriptions in the Roumieu case, show that it suffices 
to test with continuous linear functionals that detect bounded sets, and prove a uniform boundedness principle. 
The $\cD$ and $\cD^{[M]}$ exponential law is treated at the end of Section \ref{sec:BSexp}.

The $\Sp$ and $\SpM$ exponential law is treated in Section \ref{sec:Wexp}.

In Section \ref{sec:failure} we show that the $\BrM$, $\SprM$, $\SrLM$, and $\DrM$ exponential law 
fails if $M$ or $L$ have non-moderate growth. 

None of the classes $\cA$ of test functions form a category since there are no identities. 
In Section \ref{sec:comp} we prove that $\cB$ and $\BM$ are closed under composition, in contrast to all other cases. 
In fact the ``$0$th derivative'' of the composite function may not have the required decay properties at infinity.  
We show that stability under composition holds if one requests the defining properties only from the first derivative onwards.

In the final Section \ref{sec:diff} we apply the results of this paper to give a new simple proof,
in particular cases, of the fact that 
\[
  \Diff\cA=\DiffA := \big\{F=\Id+f: f\in \cA(\R^n,\R^n), \inf_{x \in \R^n}\det(\mathbb I_n + df(x)) >0\big\}
\]
is a Lie group.
It was shown in \cite{MichorMumford13} (and $\Diff\cD$ was already treated in \cite{Michor80II} 
and \cite{Michor80}) that 
the groups of diffeomorphisms ($1 \le p<q<\infty$)
\[
      \xymatrix{
        \Diff\cD \ar@{{ >}->}[r] & \Diff\cS \ar@{{ >}->}[r] & \Diff W^{\infty,p} \ar@{{ >}->}[r] & \Diff W^{\infty,q} \ar@{{ >}->}[r] & \Diff\cB 
      }
\]
are $C^\infty$-regular Lie groups. The arrows describe $C^\infty$ injective group homomorphisms.
Each group is a normal subgroup of the groups on its right.
In \cite{KrieglMichorRainer14a} we proved that, 
provided that $M=(M_k)$ is log-convex, has moderate growth, and in the Beurling $\CbM \supseteq C^\om$, 
and that $L=(L_k)$ satisfies $L_k \ge 1$ for all $k$, 
the groups of $\CM$-diffeomorphisms
\[
      \xymatrix{
        \Diff\DM \ar@{{ >}->}[r] & \Diff\SLM \ar@{{ >}->}[r] & \Diff\SpM \ar@{{ >}->}[r] & \Diff W^{[M],q} \ar@{{ >}->}[r] & \Diff\BM
      }
\]
  are $\CM$-regular Lie groups. The arrows describe $\CM$ injective group homomorphisms.
  Each group is a normal subgroup in the groups on its right.      
This was done by showing (via a careful application of Fa\`a di Bruno's formula) that $C^\infty$-curves and $\CM$-Banach plots, 
respectively, are preserved by the group operations, that is composition and inversion. 

In Section \ref{sec:diff} we use the exponential laws established in this paper to conclude in a simple way that 
$\Diff\cD$, $\Diff\cS$, $\Diff\Sp$, $\Diff\cB$ are $C^\infty$ Lie groups and 
that $\Diff\DrM$, $\Diff \SrLM$, $\Diff \SprM$, $\Diff \BrM$ are $\CrM$ Lie groups provided that 
$M=(M_k)$ is \emph{non-quasianalytic}. 
In these cases we know that it suffices to show that the group operations take $\cD$ or $\DrM$-curves to $C^\infty$ or 
$\CrM$-curves, respectively; see \cite{KMRc}.
By the exponential law \eqref{eq:exp} we may consider the $\cD$ or $\DrM$-curves in $\cA(\R^n,\R^n)$ simply as 
elements in $\cA(\R \times \R^n,\R^n)$, and thus the assertions reduce to results on composition and inversion of 
mappings in several real variables.

\paragraph{\bf Notation}
We use $\N = \N_{>0} \cup \{0\}$.
For each multi-index $\al=(\al_1,\ldots,\al_n) \in \N^n$, we write
$\al!=\al_1! \cdots \al_n!$, $|\al|= \al_1 +\cdots+ \al_n$, and 
$f^{(\al)}(x) = \p^\al f(x)=\frac{\p^{|\al|}}{\p x_1^{\al_1} \cdots \p x_n^{\al_n}} f(x)$. 
By $f^{(k)}(x) = d^k f(x)$ we mean the $k$-th order Fr\'echet derivative of $f$ at $x$,
and $d_v^k f(x) = \p_t^k|_{t=0} f(x+tv)$ denotes the $k$ times iterated directional derivative in direction $v$.

For a mapping $f : X \x Y \to Z$ we set $f^\vee: X \to Z^Y$, $f^\vee(x)(y):=f(x,y)$, and conversely,
for a mapping $g : X \to Z^Y$ we set $g^\wedge : X \x Y \to Z$, $g^\wedge(x,y) := g(x)(y)$.

For locally convex spaces $E$ let $\sB(E)$ denote the set of all
closed absolutely convex bounded subsets $B \subseteq E$.
Let $\SN(E)$ denote the collection of all continuous seminorms on $E$.
For $B \in \sB(E)$ we denote by $E_B$ the 
linear span of $B$ equipped with the Minkowski functional $\|x\|_B = \inf \{\la>0 : x \in \la B\}$. 
If $E$ is a convenient vector space, then $E_B$ is a 
Banach space. 
For $U \subseteq E$ we set $U_B := i_B^{-1}(U)$, where $i_B : E_B \to E$ is the inclusion 
of $E_B$ in $E$.
The collection of compact subsets $K \subseteq U$ is denoted by $\sK(U)$.

We denote by $E^*$ (resp.\ $E'$) the dual space of continuous (resp.\ bounded) linear functionals. $\Lin(E_1,\ldots,E_k;F)$
is the space of $k$-linear bounded mappings $E_1 \x \cdots \x E_k \to F$; if $E_i =E$ for all $i$, we also write $\Lin^k(E;F)$.  
If $E$ and $F$ are Banach spaces, then $\|~\|_{\Lin^k(E;F)}$ denotes the operator norm on $\Lin^k(E;F)$.

We subsume both the Beurling case $\bM$ and the Roumieu case $\rM$ under the symbol $[M]$.
Statements that involve more than one $[M]$ symbol must not be interpreted by mixing $\bM$ and $\rM$.

\section{Preliminaries} \label{sec:prelim}

\subsection{Weight sequences} \label{ssec:ws}

A \textbf{weight sequence} is a sequence
$M=(M_k)=(M_k)_{k=0,1,\ldots}$ of positive real numbers satisfying $M_0 = 1 \le M_1$. 

We say that $M=(M_k)$ is \textbf{log-convex} if $k\mapsto\log M_k$ is convex, or equivalently,  
\begin{equation}
  M_k^2 \le M_{k-1} M_{k+1}, \quad k \in \N.
\end{equation}
If $M=(M_k)$ is log-convex, then $M=(M_k)$ has the following properties:
\begin{align}
  &M=(M_k) \text{ is \textbf{weakly log-convex}, i.e., }  k!\, M_k \text{ is log-convex}, \\
  &(M_k)^{1/k} \text{ is non-decreasing}, \\
  &M_j  M_k\le M_{j+k}, \quad\text{for } j,k\in \N, \label{eq:alg}\\
  &M_1^j \, M_k\ge M_j\, M_{\al_1} \cdots M_{\al_j}, \quad \text{for }\al_i\in \N_{>0}, ~\al_1+\dots+\al_j = k,  \label{eq:FdB} 
\end{align}
cf.\ \cite{KMRu} or \cite{RainerSchindl12}. 

We say that $M=(M_k)$ is
\textbf{derivation closed} if 
\begin{equation} \label{eq:dc}
\sup_{k \in \N_{>0}} \Big(\frac{M_{k+1}}{M_k}\Big)^{\frac{1}{k}} < \infty,
\end{equation}
and that $M=(M_k)$ has \textbf{moderate growth} if 
\begin{equation} \label{eq:mg0}
\sup_{j,k \in \N_{>0}} \Big(\frac{M_{j+k}}{M_j \, M_k}\Big)^{\frac{1}{j+k}} <
\infty.
\end{equation}
Obviously, \eqref{eq:mg0} implies \eqref{eq:dc}. 
If $M=(M_k)$ is
derivation closed, then also $k!\, M_k$ is derivation closed and we have
\begin{equation} \label{eq:dc1}
  (k+j)!\, M_{k+j} \le C^{j(k+j)}\, k!\, M_k, \quad\text{for } k,j \in \N 
\end{equation}
for some constant $C\ge 1$.

A weakly log-convex weight sequence $M=(M_k)$ is called \textbf{quasianalytic} if 
\begin{equation}
  \sum_{k=1}^\infty (k!\, M_k)^{-1/k} = \infty,
\end{equation}
and \textbf{non-quasianalytic} otherwise. 
We refer to \cite{KMRc},\cite{KMRq}, \cite{KMRu}, or \cite{RainerSchindl12} 
for a detailed exposition of the connection between these conditions on $M=(M_k)$
and the properties of $\CM$.

\subsection{Local Denjoy--Carleman classes} \label{ssec:DC}

Let $E,F$ be Banach spaces, $U \subseteq E$ open, and let $M=(M_k)$ be a weight sequence.
We define the local Denjoy--Carleman classes 
\begin{align*}
  C^{\bM}(U,F) &:= \Big\{f \in C^\infty(U,F) : \forall K \in \sK(U) ~\forall \rh>0: 
  \|f\|_{K,\rh}^M <\infty\Big\}, \\
  C^{\rM}(U,F) &:= \Big\{f \in C^\infty(U,F) : \forall K \in \sK(U) ~\exists \rh>0: 
  \|f\|_{K,\rh}^M <\infty\Big\},
\end{align*}
where
\[
  \|f\|_{K,\rh}^M := \sup_{\substack{k \in \N\\ x \in K}} \frac{\|f^{(k)}(x)\|_{\Lin^k(E;F)}}{\rh^k k! M_k}.
\]
See \cite[4.2]{KMRu} for the locally convex structure of these spaces.
The elements of $C^{\bM}(U,F)$ are said to be of Beurling type; those of $C^{\rM}(U,F)$ of Roumieu type. 

The classes $C^{[M]}$ can be extended to 
convenient vector spaces, and they then form cartesian closed categories if the weight sequence $M=(M_k)$ 
is log-convex and has moderate growth.
This has been developed in \cite{KMRc}, \cite{KMRq}, and \cite{KMRu}.

\section{Classes of test functions between convenient vector spaces} \label{sec:convenient}

\subsection{Between Banach spaces} \label{ssec:Banach}

Let 
$E, F$ be Banach spaces, $U \subseteq E$ open.

\paragraph{\bf Smooth functions with globally bounded derivatives} 

Consider
\[
  \cB(U,F) := \Big\{f \in C^\infty(U,F) : \|f\|^{(k)}_{U} < \infty \text{ for all } k\in \N\Big\},
\]
where
\[
   \|f\|^{(k)}_{U} := \sup_{x \in U} \|f^{(k)}(x)\|_{\Lin^k(E;F)}, 
\]
with its natural Fr\'echet topology.

\paragraph{\bf Rapidly decreasing Schwartz functions}

Consider  
\[
  \cS(E,F) := \Big\{f \in C^\infty(E,F) : \|f\|^{(k,\ell)}_{E} < \infty \text{ for all } k,\ell \in \N\Big\},
\]
where 
\[
  \|f\|^{(k,\ell)}_{E} := \sup_{x \in E}  (1+\|x\|)^k \|f^{(\ell)}(x)\|_{\Lin^\ell(E;F)},
\]
with its natural Fr\'echet topology.

\paragraph{\bf Global Denjoy--Carleman classes}

Let $M=(M_k)$ be a weight sequence, and let $\rh>0$.
Consider the Banach space 
\[
  \cB^M_\rh(U,F) := \{f \in C^\infty(U,F) : \|f\|_{U,\rh}^M < \infty\},
\]
where
\[
  \|f\|_{U,\rh}^M := \sup_{\substack{k \in \N\\ x \in U}} \frac{\|f^{(k)}(x)\|_{\Lin^k(E;F)}}{\rh^k k! M_k}.
\]
We define the Fr\'echet space 
\[
  \cB^{\bM}(U,F) := \varprojlim_{n \in \N} \cB^M_{\frac{1}{n}}(U,F) 
\]
and  
\[
  \cB^{\rM}(U,F) := \varinjlim_{n \in \N} \cB^M_{n}(U,F)
\]
which is a compactly regular (LB)-space and thus {($c^\oo$-)}complete, webbed, and (ultra-)bornological; 
see \cite[Lemma 4.9]{KrieglMichorRainer14a}.

\paragraph{\bf Gelfand--Shilov classes}

Let $L=(L_k)$, $M=(M_k)$ be weight sequences, and let $\rh>0$.
Consider the Banach space 
\[
  \cS_{L,\rh}^M(E,F) := \{f \in C^\infty(E,F) : \|f\|_{E,\rh}^{L,M} < \infty\}.
\]
with the norm
\[
  \|f\|_{E,\rh}^{L,M} := \sup_{\substack{k,\ell \in \N\\ x \in E}}  \frac{(1+\|x\|)^k \|f^{(\ell)}(x)\|_{\Lin^\ell(E,F)}}{\rh^{k+\ell}\, 
  k! \ell!\, L_k M_\ell}.
\]
We define the Fr\'echet space 
\[
  \SbLM(E,F) := \varprojlim_{n \in \N} \cS^M_{L,\frac{1}{n}}(E,F) 
\]
and  
\[
  \SrLM(E,F) := \varinjlim_{n \in \N} \cS^M_{L,n}(E,F)
\]
which is a compactly regular (LB)-space and thus {($c^\oo$-)}complete, webbed, and (ultra-)bornological; 
see \cite[Lemma 4.9]{KrieglMichorRainer14a}.

\subsection{Between convenient vector spaces}  \label{ssec:convenientstructure}

For convenient vector spaces $E, F$, $c^\infty$-open $U \subseteq E$, and weight sequences $L=(L_k)$, $M=(M_k)$ we define:
\begin{empheq}[box=\widefbox]{align*}
\;\cB(U,F) &:= 
  \Bigl\{f\in C^\oo(U,F): 
  \forall \ell\; \forall B: \ell \o f \o i_B \in \cB(U_B,\R)\Bigr\} \\
  \cS(E,F) &:= 
  \Bigl\{f\in C^\oo(E,F):
  \forall \ell\; \forall B: \ell \o f \o i_B \in \cS(E_B,\R)\Bigr\} \\
  \;\BM(U,F) &:= 
  \Bigl\{f\in C^\oo(U,F):
  \forall \ell\; \forall B: \ell \o f \o i_B \in \BM(U_B,\R)\Bigr\} \\
  \;\SLM(E,F) &:= 
  \Bigl\{f\in C^\oo(E,F):
  \forall \ell\; \forall B: \ell \o f \o i_B \in \SLM(E_B,\R)\Bigr\}
\end{empheq}
where $\ell \in F^*$, $B\in \sB(E)$, and $U_B = U \cap E_B$. 
It will follow from Lemma~\ref{AvsAb} 
that for Banach spaces $E$, $F$ this definition coincides with the one given in Subsection~\ref{ssec:Banach}.

For $\cA \in \{\cB, \cS, \BM, \SLM\}$ (if $\cA\in \{\cS,\SLM\}$ we set $U=E$),
we equip $\cA(U,F)$ with the
initial locally convex structure induced by all linear mappings 
\begin{align*}
\cA(U,F) &\East{\cA(i_B,\ell)}{}  \cA(U_B, \mathbb R), \quad f \mapsto \ell\o f\o i_B.
\end{align*}
Then $\cA(U,F)$ is a
convenient vector space as $c^\oo$-closed subspace in the product
$\prod_{\ell,B} \cA(U_B, \mathbb R)$,
since smoothness can be tested by composing with the inclusions $E_B\to E$ and with the
$\ell\in F^*$, see \cite[2.14.4 and 1.8]{KM97}.
This shows at the same time, that in the definition of $\cA(U,F)$ it is not necessary to 
require that $f$ is $C^\infty$.

\subsection{Related classes defined by boundedness conditions}

Consider the collections
\begin{align*}
  \fS_{\cB} &:= \{\set^{(k)}_{U_B} : B \in \sB(E), k \in \N \},\\ 
  \fS_{\cS} &:= \{\set^{(k,\ell)}_{E_B} : B \in \sB(E), k,\ell \in \N \}, \\
  \fS_{\BbM} &:= \{\set^M_{U_B,\rh} : B \in \sB(E), \rh >0 \}, \\
  \fS_{\BrM,B} &:= \{\set^M_{U_B,\rh} : \rh >0 \}, \quad B \in \sB(E), \\
  \fS_{\SbLM} &:= \{\set^{L,M}_{E_B,\rh} : B \in \sB(E), \rh >0\}, \\
  \fS_{\SrLM,B} &:= \{\set^{L,M}_{E_B,\rh} : B \in \sB(E)\}, \quad B \in \sB(E),
\end{align*}
of set-valued mappings
\begin{align*}
    \set^{(k)}_{U_B}(f) &:= \Big\{f^{(k)}(x)(v_1,\dots,v_k) : x\in U_B,\|v_i\|_B\leq 1\Big\},\\
    \set^{(k,\ell)}_{E_B}(f) &:= \Big\{(1+\|x\|_B)^k f^{(\ell)}(x)(v_1,\dots,v_\ell) : x\in E_B,\|v_i\|_B\leq 1\Big\},\\
    \set^M_{U_B,\rh}(f) &:= \Big\{\frac{f^{(k)}(x)(v_1,\dots,v_k)}{\rh^k\, k!\, M_k}:k\in \mathbb N,x\in U_B,\|v_i\|_E\leq 1\Big\},\\
    \set^{L,M}_{E_B,\rh}(f) &:= \Big\{\frac{(1+\|x\|_B)^k f^{(\ell)}(x)(v_1,\dots,v_\ell)}{\rh^{k+\ell}\,k! \ell!\, L_k M_\ell}:k,\ell \in \mathbb N,x\in E_B,\|v_i\|_B\leq 1\Big\}.
\end{align*}
For $C^\infty$-mappings $f$ we define
\begin{align*} 
f \in \cBb(U,F) ~ &:\Longleftrightarrow ~
  ~\forall \set \in \fS_{\cB} : 
  \set(f) \text{ is bounded in } F, \\
f \in \cSb(E,F) ~ &:\Longleftrightarrow ~ 
  ~\forall \set \in \fS_{\cS} : 
  \set(f) \text{ is bounded in } F, \\
f \in \BbM_b(U,F) ~ &:\Longleftrightarrow ~
  ~\forall \set \in \fS_{\BbM} : 
  \set(f) \text{ is bounded in } F, \\
f \in \BrM_b(U,F) ~ &:\Longleftrightarrow ~
  ~\forall B \in \sB(E) ~\exists \set \in \fS_{\BrM,B} : 
  \set(f) \text{ is bounded in } F, \\  
f \in \SbLMb(E,F) ~ &:\Longleftrightarrow ~ 
  ~\forall \set \in \fS_{\SbLM} : 
  \set(f) \text{ is bounded in } F, \\  
f \in \SrLMb(E,F) ~ &:\Longleftrightarrow ~
  ~\forall B \in \sB(E) ~\exists \set \in \fS_{\SrLM,B} : 
  \set(f) \text{ is bounded in } F.  
\end{align*}

Moreover, we call a subset $\mathcal F$ of such functions $f$
bounded in the corresponding space, when the conditions above
are satisfied for $\bigcup_{f\in\mathcal F}\set(f)$ instead of $\set(f)$.

\begin{lemma} \label{AvsAb}
  We always have 
  \begin{align}  \label{eq:AvsAb1}
  \begin{split}
    \cBb(U,F)&=\cB(U,F),\\
    \cSb(E,F)&=\cS(E,F),\\
    \BbM_b(U,F)&=\BbM(U,F),\\
    \SbLMb(E,F)&=\SbLM(E,F).
  \end{split}
  \end{align}
  We have  
  \begin{align} \label{eq:AvsAb2}
  \begin{split}
    \BrM_b(U,F)&=\BrM(U,F),\\
    \SrLMb(E,F)&=\SrLM(E,F),
  \end{split}  
  \end{align}
  if there exists a Baire vector space topology on the dual $F^*$ for which $\on{ev}_x$ is continuous for all $x \in F$. 

	Moreover, the bounded sets of both sides of the equalities are the same.
\end{lemma}

\begin{proof}
  Let $f : E \supseteq U \to F$ be $C^\infty$. Then, for $\cA \in \{\cB,\cS,\BbM,\SbLM\}$,
    \begin{align*}
      f \in \cA(U,F) 
      &\Longleftrightarrow  
      \forall \ell \in F^* ~\forall \set \in \fS_\cA: \set(\ell \o f) \text{ is bounded in } \R\\
       &\Longleftrightarrow  
       \forall \set \in \fS_\cA ~\forall \ell \in F^*: \ell(\set(f)) \text{ is bounded in } \R\\ 
       &\Longleftrightarrow 
       \forall \set \in \fS_\cA: \set(f) \text{ is bounded in } F\\ 
       &\Longleftrightarrow f \in \cAb(U,F),  
    \end{align*}
    which shows \eqref{eq:AvsAb1}.
	The same argument with some obvious changes shows the equivalence of the boundedness of a set $\mathcal F\subs C^\oo(U,F)$
	in $\cA(U,F)$ and in $\cAb(U,F)$.

  Let $f \in \cA(U,F)$ for $\cA \in \{\BrM,\SrLM\}$.
    Fix $B \in \sB(E)$ and, for $\set \in \fS_{\cA,B}$ and $C>0$, 
    consider the sets
    \[
    A_{\set,C} :=\Big\{\ell\in F^*: |y| \le C \text{ for all } y \in \set(\ell \o f)\}\Big\}
    \]
    which are closed subsets in $F^*$ for the given Baire topology. 
    We have
    $\bigcup_{\set,C}A_{\set,C}= F^*$ and by the Baire property there exist $\set$
    and $C$ such that the interior $\on{int}(A_{\set,C})$ of $A_{\set,C}$ is non-empty. If
    $\ell_0\in \on{int}(A_{\set,C})$, then for each $\ell\in F^*$ there is a $\de>0$ such that 
    $\de\ell\in \on{int}(A_{\set,C})-\ell_0$, and, hence, since
    \begin{align*}
    \de |(\ell\o f)^{(k)}(x)(v_1,\ldots)| &\le |((\de\,\ell+\ell_0)\o f)^{(k)}(x)(v_1,\dots)| +
    |(\ell_0\o f)^{(k)}(x)(v_1,\dots)|
    \end{align*}
    we conclude that the set $\set(\ell \o f)$ is bounded (by $2C/\de$). 
    So the set $\set(f)$ is bounded, and thus $f \in \cAb(U,F)$. 

    To see that a set $\cF \subseteq C^\infty(U,F)$ which is bounded in $\cA(U,F)$ is also bounded in 
    $\cAb(U,F)$ it suffices to repeat the argument with 
    $A_{\set,C} :=\{\ell\in F^*: |y| \le C \text{ for all } y \in \bigcup_{f\in\mathcal F}\set(\ell \o f)\}\}$.
\end{proof}

\begin{proposition}[{\cite[Prop.~5.1]{KrieglMichorRainer14a}}] \label{prop:incl1}
  Let $M=(M_k)$, $L=(L_k)$ be weight sequences, and 
  let $E,F$ be convenient vector spaces.
  We have the following inclusions. 
    \[
      \xymatrix{
        \cS(E,F) \ar@{{ >}->}[r]  & \cB(E,F) \ar@{{ >}->}[r] & C^\infty(E,F) \\
        \SrLM(E,F) \ar@{{ >}->}[r] \ar@{{ >}->}[u]  & \BrM(E,F) \ar@{{ >}->}[u] \ar@{{ >}->}[r] & \CrM(E,F) \ar@{{ >}->}[u]\\
        \SbLM(E,F) \ar@{{ >}->}[r] \ar@{{ >}->}[u]  & \BbM(E,F) \ar@{{ >}->}[u] \ar@{{ >}->}[r] & \CbM(E,F) \ar@{{ >}->}[u]
      }
    \]
\end{proposition}

\section{Working up to the exponential law} \label{sec:workingup}

\subsection{Projective descriptions in the Roumieu cases} \label{sec:proj}

We define
\begin{align*} \label{eq:sR}
\begin{split}
  \sR &:= \{(r_k)_{k\in \N} \subseteq \R_{>0} :  r_k \si^k \to 0 \text{ for all }\si>0 \} \\
  \sR' &:= \{(r_k) \in \sR :  r_kr_\ell \ge r_{k+\ell} \text{ for all } k,\ell\in \N\}.
\end{split}
\end{align*}

\begin{lemma*} 
  Let $a_\al \ge 0$ for $\al \in \N^m$.
  The following are equivalent:
  \begin{align}
    \exists \si>0 &: \sup_\al \frac{a_\al}{\si^{|\al|}} < \infty \\
    \forall (r_k) \in \sR &: \sup_\al r_{|\al|} a_\al< \infty \\
    \forall (r_k) \in \sR'~\exists \de>0 &: \sup_\al \de^{|\al|} r_{|\al|} a_\al< \infty 
  \end{align}
\end{lemma*}

\begin{proof}
   Set $b_k := \max_{|\al|=k} a_\al$.

  \thetag{1} $\Rightarrow$ \thetag{2}
  There exists $\si>0$ such that 
  \begin{align*}
    r_{|\al|} a_{\al} = r_{|\al|} \si^{|\al|} (a_{\al}/\si^{|\al|}) 
  \end{align*}
  is bounded uniformly in $\al \in \N^m$.
  
  \thetag{2} $\Rightarrow$ \thetag{3} Use $\de=1$.
  
  \thetag{3} $\Rightarrow$ \thetag{1} 
  For $(r_k) \in \sR'$ there exists $\de>0$
  so that
  \[
   \sup_{k\in \N} \de^k\, r_k\, b_k = \sup_{k\in \N} \max_{|\al|=k} \de^{|\al|}\, r_{|\al|}\, a_{\al} 
   = \sup_{\al \in \N^m} \de^{|\al|}\, r_{|\al|}\, a_{\al} < \infty. 
  \]  
  By
  \cite[9.2(4$\Rightarrow$1)]{KM97} 
  the formal power series $\sum_{k\ge 0} b_k t^k$ has positive 
  radius of convergence. Thus $(b_k/\si^k)_k$ and hence also $(a_{\al}/\si^{|\al|})_{\al}$ is bounded for some $\si>0$.
  This implies \thetag{1}.
\end{proof}

For a $C^\infty$-mapping $f : E \supseteq U \to F$ between Banach spaces and a positive sequence $(r_k)$ consider
\begin{align*}
  \set^M_{U,(r_k)}(f) &:= \Big\{r_k \frac{f^{(k)}(x)(v_1,\dots,v_k)}{k!\, M_k}  : k\in \mathbb N,x\in U,\|v_i\|\leq 1\Big\} \\
  \set^{L,M}_{E,(r_k)}(f) &:= \Big\{r_{k+\ell} \frac{(1+\|x\|)^k f^{(\ell)}(x)(v_1,\dots,v_\ell)}{k! \ell!\, L_k M_\ell}  : 
    k,\ell \in \N,x\in E,\|v_i\|\leq 1\Big\}.
\end{align*}
In particular, for $\si>0$ we have $\set^M_{U,(\si^{-k})}(f) = \set^M_{U,\si}(f)$ and 
$\set^{L,M}_{E,(\si^{-k})}(f) = \set^{L,M}_{E,\si}(f)$. 
Define 
\begin{align*}
  \fS_{\BrM}^R &:= \{\set^M_{U,(r_k)} : (r_k) \in \sR\}, \\ 
  \fS_{\SrLM}^R &:= \{\set^{L,M}_{E,(r_k)} : (r_k) \in \sR\}.
\end{align*}

\begin{proposition*} 
For a $C^\infty$-mapping $f : E \supseteq U \to F$ between Banach spaces $E$ and $F$ 
the following are equivalent.
\begin{enumerate}
\item $f$ is $\cB^{\rM}=\cBb^{\rM}$.
\item For each $\set \in \fS_{\BrM}^R$, 
the set $\set(f)$ is bounded in $F$.
\item For each sequence $(r_k) \in \sR'$ 
there exists 
$\de>0$ such that the set $\set^M_{U,(r_k\, \de^k)}(f)$
is bounded in $F$. 
\end{enumerate}
Moreover, the following are equivalent.
\begin{enumerate}
  \item $f$ is $\SrLMb= \SrLM$.
  \item For each $\set \in \fS_{\SrLM}^R$, 
  the set $\set(f)$ is bounded in $F$.
  \item For each sequence $(r_k) \in \sR'$  
  there exists 
  $\de>0$ such that the set $\set^{L,M}_{E,(r_k\, \de^k)}(f)$
  is bounded in $F$. 
  \end{enumerate}
\end{proposition*}

\begin{proof}
  For $\cA=\BrM$ set
  \begin{align*}
  a_k:=\sup_{x\in U}\frac{\|f^{(k)}(x)\|_{\Lin^k(E;F)}}{k!\, M_k},
\end{align*}
and for $\cA=\SrLM$ set
\begin{align} \label{eq:projS} \tag{4}
  a_{k,\ell}:=\sup_{x\in E}\frac{(1+\|x\|)^k \|f^{(\ell)}(x)\|_{\Lin^\ell(E;F)}}{k!\ell!\,L_k M_\ell},
\end{align}
and apply Lemma \ref{sec:proj}.
\end{proof}

\subsection{Testing with bounded linear functionals that detect bounded sets} \label{lem:Bdetect}

\begin{lemma*} 
Let $E$ be a Banach space, let $U \subseteq E$ be open, 
and let $F$ be a convenient vector space.
Let $\sS$ be a
family of bounded linear functionals on $F$ which together detect bounded
sets (i.e., $B\subseteq F$ is bounded if and only if $\ell(B)$ is bounded for all
$\ell\in\sS$). Then:
\begin{align*}
  f \in \cB(U,F) ~&\Longleftrightarrow~  \ell\o f \in \cB(U,\R) \text{ for all } \ell \in \sS, \\
  f \in \cS(E,F) ~&\Longleftrightarrow~  \ell\o f \in \cS(E,\R) \text{ for all } \ell \in \sS, \\
  f \in \BM(U,F) ~&\Longleftrightarrow~  \ell\o f \in \BM(U,\R) \text{ for all } \ell \in \sS, \\
  f \in \SLM(E,F) ~&\Longleftrightarrow~  \ell\o f \in \SLM(E,\R) \text{ for all } \ell \in \sS.
\end{align*}
\end{lemma*}

\begin{proof}
For $C^\infty$-curves this follows from \cite[2.1 and 2.11]{KM97}, and, by composing with such, 
it follows for $C^\infty$-mappings $f: U \to F$. 

For $\cA\in \{\cB,\cS,\BbM,\SbLM\}$ we have, by Lemma \ref{AvsAb},  
\begin{align*}
  f \in \cA(U,F) ~&\Longleftrightarrow ~f \in \cAb(U,F) \\
  &\Longleftrightarrow ~\forall \set \in \fS_\cA: ~\set(f) \text{ is bounded in } F \\
  &\Longleftrightarrow ~\forall\ell \in \sS ~\forall \set \in \fS_\cA: ~\ell(\set(f)) = \set(\ell \o f) \text{ is bounded in } \R
\end{align*}
since $\sS$ detects bounded sets.

For $\cA\in \{\BrM,\SrLM\}$ we have, by Proposition \ref{sec:proj},
\begin{align*}
  f \in \cA(U,F) ~&\Longleftrightarrow ~\forall \ell \in F^* : \ell\o f \in \cA(E,\R) \\
  &\Longleftrightarrow ~\forall \ell \in F^* ~\forall \set \in \fS_\cA^R : ~\set(\ell \o f) \text{ is bounded in } \R \\
  &\Longleftrightarrow ~\forall \set \in \fS_\cA^R : ~\set(f) \text{ is bounded in } F \\
  &\Longleftrightarrow ~\forall \ell \in \sS ~\forall \set \in \fS_\cA^R : ~\set(\ell \o f) \text{ is bounded in } \R  
\end{align*}
since $\sS$ detects bounded sets.
\end{proof}

\subsection{The uniform boundedness principle} \label{thm:uniformboundedness}

\begin{theorem*} 
Let $E$, $F$, $G$ be convenient vector spaces and let $U \subseteq F$ be $c^\infty$-open. 
A linear mapping $T:E\to \cB(U,G)$, $T:E\to \cS(F,G)$, $T:E\to \BM(U,G)$, or $T:E\to \SLM(F,G)$
is bounded if and only if $\on{ev}_x\o T: E\to G$ is bounded for every $x\in U$ or $x \in F$, respectively. 
\end{theorem*}

\begin{proof}
  $(\Rightarrow)$   
  For $x\in U$ and $\ell\in G^*$, the linear mapping $\ell\o\on{ev}_x = \cA(x,\ell):\cA(U,G)\to
  \R$ is continuous, thus $\on{ev}_x$ is bounded. 
  Therefore, if $T$ is bounded then so is $\on{ev}_x\o T$.
  
  $(\Leftarrow)$
  Suppose that $\on{ev}_x\o T$ is bounded for all $x\in U$.
  By definition it is enough to show that $T$ is bounded for Banach spaces $E$, $F$, and $G=\RR$
  which follows from the 
  closed graph theorem \cite[52.10]{KM97}, 
  as $\cA(U,\R)$ is a Fr\'echet space if $\cA \in \{\cB,\cS,\BbM\}$ or a compactly regular (LB)-space and thus webbed if 
  $\cA \in \{\BrM,\SrLM\}$, see Subsection~\ref{ssec:Banach}.
\end{proof}

\section{The \texorpdfstring{$\cB$}{B}, \texorpdfstring{$\cS$}{S},  \texorpdfstring{$\cD$}{D}, 
\texorpdfstring{$\BM$}{BM}, \texorpdfstring{$\SLM$}{SLM}, and \texorpdfstring{$\DM$}{DM} exponential law} \label{sec:BSexp}

First we prove the exponential law for \texorpdfstring{$\cB$}{B}, \texorpdfstring{$\cS$}{S},  
\texorpdfstring{$\BM$}{BM}, and \texorpdfstring{$\SLM$}{SLM}. At the end of the section we obtain the $\cD$ and the $\DM$ 
exponential law as an application of the $\cB$ and the $\BM$ case, respectively.  

\begin{theorem} \label{thm:Bexp}
  Let $L=(L_k)$, $M=(M_k)$ be weakly log-convex weight sequences with moderate growth.
  For convenient vector spaces $E_1, E_2, F$ and $c^\infty$-open $U_i \subseteq E_i$ we have the exponential law:
  \begin{align*} 
    \cB(U_1 \times U_2,F) &= \cB(U_1,\cB(U_2,F)), \\ 
    \cS(E_1 \times E_2,F) &= \cS(E_1,\cS(E_2,F)), \\
    \BM(U_1 \times U_2,F) &= \BM(U_1,\BM(U_2,F)), \\
    \SLM(E_1 \times E_2,F) &= \SLM(E_1,\SLM(E_2,F))
  \end{align*}
\end{theorem}

\begin{remark*}
  In the $\BM$-exponential law the inclusion \thetag{$\supseteq$} holds 
  without $M=(M_k)$ having moderate growth, 
  the inclusion \thetag{$\subseteq$}  
  without $M=(M_k)$ being weakly log-convex. 
  The analogous statement holds for the $\SLM$-exponential law, where the inclusions 
  hold without the respective  
  conditions for $M=(M_k)$ \emph{and} $L=(L_k)$. 
\end{remark*}

\begin{proof}
Let $\cA \in \{\cB,\cS,\BM,\SLM\}$.
We have the $C^\infty$-exponential law $C^{\infty}(U_1 \times U_2,F) \cong C^\infty(U_1,C^\infty(U_2,F))$, by \cite[3.12]{KM97}; 
thus, in the following all mappings are assumed to be smooth.   
We have the following equivalences, where $B \in \sB(E_1 \times E_2)$ and  
$B_i \in \sB(E_i)$.  
  \begin{multline*}
    f \in \cA(U_1 \times U_2,F) 
    \Longleftrightarrow  
    \forall \ell \in F^* ~\forall B: \ell \o f \o i_B \in \cA((U_1 \times U_2)_B,\R)\\
    \Longleftrightarrow  
    \forall \ell \in F^* ~\forall B_1,B_2: \ell \o f \o (i_{B_1} \times i_{B_2}) \in \cA((U_1)_{B_1} \times (U_2)_{B_2},\R)  
  \end{multline*}
For the second equivalence we use that every bounded $B \subseteq E_1 \times E_2$ is contained in $B_1 \times B_2$ for some bounded $B_i \subseteq E_i$, 
and, thus, the inclusion $(E_1 \times E_2)_B \to (E_1)_{B_1} \times (E_1)_{B_2}$ is bounded.
On the other hand,
\begin{multline*}
  f^\vee \in \cA(U_1,\cA(U_2,F)) 
  \Longleftrightarrow  
  \forall B_1: f^\vee \o i_{B_1} \in \cA((U_1)_{B_1},\cA(U_2,F))\\
  \Longleftrightarrow  
  \forall \ell \in F^* ~\forall B_1,B_2: \cA(i_{B_2},\ell) \o f^{\vee} \o i_{B_1} \in \cA((U_1)_{B_1},\cA((U_2)_{B_2},\R))   
\end{multline*}   
For the second equivalence we use Lemma~\ref{lem:Bdetect} and the fact that the linear mappings $\cA(i_{B_2},\ell)$ generate the bornology. 
These considerations imply that we may restrict to Banach spaces $E_i$ and $F = \R$.

\paragraph{{\bf Direction} ($\Rightarrow$)} 
Let $f \in \cA(U_1 \times U_2,\R)$. It is clear that $f^\vee$ takes values in $\cA(U_2,\R)$. 
Moreover, the mapping $f^\vee:U_1\to \cA(U_2,\RR)$ is $C^\oo$ with $d^jf^\vee=(\d_1^jf)^\vee$; this can be proved 
in the same way as the claim in \cite[5.2]{KMRu}.
We have to show that 
\begin{equation} \label{eq:toshow1}
  f^\vee : U_1 \to \cA(U_2,\R) \text{ is } \cA=\cAb.   
\end{equation}

\paragraph{\bf Case $\cA \in \{\cB,\cS\}$}
For $\cA = \cB$, \eqref{eq:toshow1}
is equivalent to 
\begin{align}
  &~\forall k_1 \in \N : ~\set^{(k_1)}_{U_1}(f^\vee) \text{ is bounded in } \cB(U_2,\R) \nonumber\\
  \Longleftrightarrow
  &~\forall k_1,k_2 \in \N : ~\sup\Big\{\|y\|^{(k_2)}_{U_2} : y \in \set^{(k_1)}_{U_1}(f^\vee) \Big\} < \infty \nonumber\\
  \Longleftrightarrow  \label{eq:Bexp2}
  &~\forall k_1,k_2 \in \N : ~\sup_{\substack{x_i \in U_i\\ \|v_j^i\|_{E_i} \le 1}} 
  |\p_2^{k_2}\p_1^{k_1} f(x_1,x_2)(v_1^1,\dots,v_{k_1}^1;v_1^2,\dots,v_{k_2}^2)|
   < \infty, 
\end{align}
for $\cA=\cS$ to 
\begin{align}
  &~\forall k_1,\ell_1 \in \N : ~\set^{(k_1,\ell_1)}_{U_1}(f^\vee) \text{ is bounded in } \cS(U_2,\R) \nonumber\\
  \Longleftrightarrow
  &~\forall k_1,k_2,\ell_1,\ell_2 \in \N : ~\sup\Big\{\|y\|^{(k_2,\ell_2)}_{U_2} : y \in \set^{(k_1,\ell_1)}_{U_1}(f^\vee) \Big\} < \infty \nonumber\\
  \Longleftrightarrow \nonumber
  &~\forall k_1,k_2,\ell_1,\ell_2 \in \N :\\  \label{eq:Sexp2}
  &~\sup_{\substack{x_i \in U_i\\ \|v_j^i\|_{E_i} \le 1}} 
  (1+\|x_1\|_{E_1})^{k_1} (1+\|x_2\|_{E_2})^{k_2} |\p_2^{\ell_2}\p_1^{\ell_1} f(x_1,x_2)(v_1^1,\dots;v_1^2,\dots)|
   < \infty 
\end{align}
which is true as $f \in \cA(U_1\times U_2,\R)$ by assumption (and by the polarization formula \cite[7.13.1]{KM97}).

\paragraph{\bf Case $\cA \in \{\BbM,\SbLM\}$}

We prove the case $\cA = \SbLM$.
The following arguments also give a proof for $\cA=\BbM$ if we 
  set $k_i \equiv 0$ and take the suprema over $x_i \in U_i$. 
\begin{equation*} 
\begin{split}
  \xymatrix{
  \cB^{\bM}(U_2,\R )\ar@{=}[r]  &
\varprojlim _{\rh _2} \cB^{M}_{\rh _2}(U_2,\R)
\ar@{->}[r]^(.7){\ell} \ar@{->}[d]  &\R  \\
    U_1\ar@{->}[u]^{f^\vee} 
 \ar@{.>}[r] &\cB^{M}_{\rh _2}(U_2,\R )
\ar@{.>}[ur]  &  \\
}  
\end{split}  
\end{equation*}
\begin{equation*} 
\begin{split}
  \xymatrix{
  \SbLM(E_2,\R )\ar@{=}[r]  &
\varprojlim_{\rh_2} \cS^{M}_{L,\rh_2}(E_2,\R)
\ar@{->}[r]^(.7){\ell} \ar@{->}[d]  &\R  \\
    E_1\ar@{->}[u]^{f^\vee} 
 \ar@{.>}[r] &\cS^{M}_{L,\rh _2}(E_2,\R )
\ar@{.>}[ur]  &  \\
}  
\end{split}  
\end{equation*}
By Lemma~\ref{lem:Bdetect}, it suffices to show that $f^\vee : E_1 \to \cS^{M}_{L,\rh_2}(E_2,\R)$ is 
$\SbLM= \SbLMb$ for each $\rh_2>0$, 
since every $\ell \in \SbLM(E_2,\R)^*$ factors over some $\cS^{M}_{L,\rh_2}(E_2,\R)$.
Thus it suffices to prove that, for all $\rh_1,\rh_2>0$, 
the set $\set^{L,M}_{E_1,\rh_1}(f^\vee)$ 
is bounded in $\cS^{M}_{L,\rh_2}(E_2,\R)$, or, equivalently, 
\begin{equation} \label{eq:SM2}
  \sup_{\substack{x_i \in E_i\\ k_i,\ell_i \in \N\\ \|v_j^i\|_{E_i} \le 1}} 
  \frac{(1+\|x_2\|_{E_2})^{k_2} (1+\|x_1\|_{E_1})^{k_1}  |\p_2^{\ell_2}\p_1^{\ell_1} f(x_1,x_2)(v_1^1,\dots,v_{\ell_1}^1;v_1^2,\dots,v_{\ell_2}^2)|}
  {\rh_2^{k_2+\ell_2}\,\rh_1^{k_1+\ell_1}\, k_2!\,k_1!\, \ell_2!\,\ell_1!\, L_{k_2}\, L_{k_1}\,  M_{\ell_2}\, M_{\ell_1}} < \infty. 
\end{equation}
Since $L=(L_k)$ and $M=(M_k)$ have moderate growth (\ref{ssec:ws}.\ref{eq:mg0}), i.e., 
\begin{equation} \label{eq:mg}
    L_{k_1+k_2} \le \ta^{k_1+k_2} L_{k_1} L_{k_2} \text{ and }
  M_{\ell_1+\ell_2} \le \ta^{\ell_1+\ell_2} M_{\ell_1} M_{\ell_2} \text{ for some } \ta>0,
\end{equation}
using $\binom{a}{b} \le 2^a$, and setting $\rh := \tfrac{1}{2 \ta} \min \{\rh_1,\rh_2\}$,
the left-hand side of \eqref{eq:SM2} is majorized by 
\begin{align} \label{eq:SM3}
  \sup_{\substack{x_i \in E_i\\ k_i,\ell_i \in \N\\ \|v_j^i\|_{E_i} \le 1}}  
  \frac{(1+\|x_2\|_{E_2})^{k_2} (1+\|x_1\|_{E_1})^{k_1}  |\p_2^{\ell_2}\p_1^{\ell_1} f(x_1,x_2)(v_1^1,\dots,v_{\ell_1}^1;v_1^2,\dots,v_{\ell_2}^2)|}
  {\rh^{k_1+k_2+\ell_1+\ell_2}\, (k_1+k_2)!\, (\ell_1+\ell_2)!\, L_{k_1+k_2}\,  M_{\ell_1+\ell_2}}   
\end{align} 
which is finite as $f \in \SbLM(E_1 \times E_2,\R)$ by assumption.

\paragraph{\bf Case $\cA \in \{\BrM,\SrLM\}$} 
We prove the case $\cA = \SrLM$.
The following arguments also give a proof for $\cA=\BrM$ if we 
  set $k_i \equiv 0$ and take the suprema over $x_i \in U_i$. 
\begin{equation*} 
\begin{split}
\xymatrix{
 \cB^{\rM}(U_2,\R )\ar@{=}[r]  &
\varinjlim _{\rh _2}\cB^{M}_{\rh _2}(U_2,\R )
\ar@{->}[r]^(0.7){\ell}   &\R  \\
   U_1\ar@{->}[u]^{f^\vee} \ar@{.>}[r]  
 &
\cB^{M}_{\rh _2}(U_2,\R )\ar@{->}[u]  &  \\
}
\end{split}  
\end{equation*}
\begin{equation*} 
\begin{split}
\xymatrix{
 \SrLM(E_2,\R )\ar@{=}[r]  &
\varinjlim _{\rh _2}\cS^{M}_{L,\rh _2}(E_2,\R )
\ar@{->}[r]^(0.7){\ell}   &\R  \\
   E_1\ar@{->}[u]^{f^\vee} \ar@{.>}[r]  
 &
\cS^{M}_{L,\rh _2}(E_2,\R )\ar@{->}[u]  &  \\
}
\end{split}  
\end{equation*}
We show that $f^\vee : E_1 \to \varinjlim_{\rh_2} \cS^{M}_{L,\rh_2}(E_2,\R)$ 
is $\SrLMb \subs \SrLM$.
It suffices to prove that there exist $\rh_1,\rh_2>0$ 
such that $\set^{L,M}_{E_1,\rh_1}(f^\vee)$ 
is bounded in $\cS^{M}_{L,\rh_2}(E_2,\R)$, or, equivalently,
\eqref{eq:SM2} holds.
Since $f \in \SrLM(E_1 \times E_2,\R)$, there exists $\rh>0$ so that \eqref{eq:SM3} is finite. 
Setting $\rh_i := 2 \ta \rh$ we have again that the left-hand side of \eqref{eq:SM2} is majorized by 
\eqref{eq:SM3}.

\paragraph{{\bf Direction} ($\Leftarrow$)}
Let $f^\vee : U_1 \to \cA(U_2,\R)$ be $\cA$. 
Then $f^\vee:U_1\to \cA(U_2,\R)\to C^\oo(U_2,\R)$
is $C^\oo$, since the latter inclusion is evidently bounded.

\paragraph{\bf Case $\cA \in \{\cB,\cS\}$} 
That $f^\vee : U_1 \to \cA(U_2,\R)$ is $\cA=\cAb$ implies \eqref{eq:Bexp2} or \eqref{eq:Sexp2}, 
respectively,
and hence $f \in \cA(U_1 \times U_2,\R)$, since 
\begin{equation} \label{eq:pd}
  d^k f(x_1,x_2) = \on{sym}\Big( \sum_{k_1+k_2=k} \p_1^{k_1} \p_2^{k_2} f(x_1,x_2)\Big),
\end{equation}
where $\on{sym}$ denotes symmetrization of multilinear mappings, and, if $\cA=\cS$, using
\begin{align} \label{eq:n}
  \sum_{k_1+k_2=k} (1+a_1)^{k_1} (1+a_2)^{k_2} &\ge 2^{-k} \sum_{k_1+k_2=k} \binom{k}{k_1} (1+a_1)^{k_1} (1+a_2)^{k_2} \nonumber\\
  &= 2^{-k}\, (2+a_1+a_2)^k \nonumber\\
  &\ge  2^{-k}\, (1+a_1+a_2)^k, \quad (a_1,a_2\ge 0), 
\end{align}
for $a_1=\|x_1\|_{E_1}$ and $a_2=\|x_2\|_{E_2}$ and choosing the Banach norm
\begin{equation} \label{eq:productnorm}
  \|(x_1,x_2)\|_{E_1 \times E_2} := \|x_1\|_{E_1}+\|x_2\|_{E_2}   
\end{equation} 
on $E_1 \times E_2$. 
Note that, if $a_{k_1,k_2} \ge 0$ and $b_k := \sum_{k_1+k_2=k} a_{k_1,k_2}$, 
then
\begin{equation} \label{eq:abelem}
   a_{k_1,k_2} \le b_k \le (k+1)\, \max_{k_1+k_2=k} a_{k_1,k_2}, \quad k=k_1+k_2.    
\end{equation} 

\paragraph{\bf Case $\cA \in \{\BbM,\SbLM\}$}
As before we prove the case $\cA=\SbLM$; the case $\cA=\BbM$ follows from the same arguments.

For each $\rh_2>0$, the mapping $f^\vee : E_1 
\to \cS^{M}_{L,\rh_2}(E_2,\R)$ is $\SbLM= \SbLMb$.
So for all $\rh_1,\rh_2>0$
the set $\set^{L,M}_{E_1,\rh_1}(f^\vee)$ is bounded in $\cS^{M}_{L,\rh_2}(E_2,\R)$ and hence \eqref{eq:SM2} holds.
Since $L=(L_k)$ and $M=(M_k)$ are weakly log-convex, we have (cf.\ (\ref{ssec:ws}.\ref{eq:alg}))
\begin{equation} \label{eq:wlc}
    k_1!\, k_2!\, L_{k_1} L_{k_2} \le (k_1+k_2)!\, L_{k_1+k_2} \text{ and }
  \ell_1!\, \ell_2!\, M_{\ell_1} M_{\ell_2} \le (\ell_1+\ell_2)!\, M_{\ell_1+\ell_2},
\end{equation}
the left-hand side of \eqref{eq:SM2} majorizes 
\begin{align} \label{eq:SM4}
  \sup_{\substack{x_i \in E_i\\ k_i,\ell_i \in \N\\ \|v_j^i\|_{E_i} \le 1}} 
  \frac{(1+\|x_2\|_{E_2})^{k_2} (1+\|x_1\|_{E_1})^{k_1}  |\p_2^{\ell_2}\p_1^{\ell_1} f(x_1,x_2)(v_1^1,\dots,v_{\ell_1}^1;v_1^2,\dots,v_{\ell_2}^2)|}
  {\rh_2^{k_2+\ell_2}\,\rh_1^{k_1+\ell_1}\, (k_1+k_2)!\, (\ell_1+\ell_2)!\, L_{k_1+k_2}\, M_{\ell_1+\ell_2}}.     
\end{align}
This implies the statement, using \eqref{eq:pd} and \eqref{eq:n}
for $a_1=\|x_1\|_{E_1}$ and $a_2=\|x_2\|_{E_2}$ and choosing the Banach norm \eqref{eq:productnorm}
on $E_1 \times E_2$.
In the situation of \eqref{eq:abelem} we have 
\begin{equation}
  \sup_{k \in \N} \frac{b_{k}}{(2\rh)^{k}} \le \sup_{k_1,k_2 \in \N} \frac{a_{k_1,k_2}}{\rh^{k_1+k_2}} \le \sup_{k \in \N} \frac{b_{k}}{\rh^{k}}.
\end{equation}

\paragraph{\bf Case $\cA \in \{\BrM,\SrLM\}$}
We prove the case $\cA=\SrLM$; the case $\cA=\BrM$ follows from the same arguments.

The inductive limit $\varinjlim _{\rh _2}\cS^{M}_{L,\rh _2}(E_2,\R )$ is compactly regular, by Lemma~\ref{ssec:Banach}.
So the dual space $(\varinjlim _{\rh _2}\cS^{M}_{L,\rh _2}(E_2,\R ))^*$ 
can be equipped with the Baire topology of the countable limit 
$\varprojlim_{\rh_2} \cS^{M}_{L,\rh _2}(E_2,\R )^*$ of Banach spaces. 
Thus $f^\vee : E_1 \to \varinjlim_{\rh_2} \cS^{M}_{L,\rh_2}(E_2,\R)$ is $\SrLMb$, by Lemma \ref{AvsAb}.
By regularity, there exists $\rh_1 >0$ so that the set $\set^{L,M}_{E_1,\rh_1}(f^\vee)$ 
is contained and bounded in
$\cS^{M}_{L,\rh_2}(E_2,\R)$ for some $\rh_2>0$.
Then the proof can be finished as in the Beurling case.  
\end{proof}

Let us show that the identities in Theorem \ref{thm:Bexp} are bornological isomorphisms.
Note that we cannot simply conclude boundedness of the mappings 
$$\cA(U_1 \times U_2,F)  \rightleftarrows  \cA(U_1,\cA(U_2,F))$$
from the exponential law as in the $C^\infty$ case \cite[3.13]{KM97} or the $\CM$ case \cite[5.5]{KMRu}.
The reason is that no linear mapping except $0$ belongs to $\cA$.

\begin{theorem}\label{thm-7.2}
  Let $L=(L_k)$ and $M=(M_k)$ be weakly log-convex and have moderate growth.
  For convenient vector spaces $E_1$, $E_2$, and $F$ and $c^\infty$-open subsets $U_i \subseteq E_i$, we have 
  bornological isomorphisms:
  \begin{align*} 
    \cB(U_1 \times U_2,F) &\cong \cB(U_1,\cB(U_2,F)), \\ 
    \cS(E_1 \times E_2,F) &\cong \cS(E_1,\cS(E_2,F)), \\
    \BM(U_1 \times U_2,F) &\cong \BM(U_1,\BM(U_2,F)), \\
    \SLM(E_1 \times E_2,F) &\cong \SLM(E_1,\SLM(E_2,F)).
  \end{align*}
\end{theorem}

\begin{proof}
  This is a consequence of the uniform boundedness principle, Theorem \ref{thm:uniformboundedness}.  

  First we check that 
  the mapping
  \begin{equation} \label{eq:Biso4}
    \cA(U_1 \times U_2,F) \ni f \mapsto  f^\vee(x) \in \cA(U_2,F)
  \end{equation}
  is bounded for each $x \in U_1$. By definition we may suppose that $E_i$ and $F$ are Banach spaces, in fact:
  \[
    \xymatrix{
    \cA(U_1 \times U_2,F) \ar[rr] \ar[d]_{\cA(i_{B_1} \x i _{B_2},\ell)} && \cA(U_2,F) \ar[d]^{\cA(i _{B_2},\ell)} \\
    \cA((U_1)_{B_1} \times (U_2)_{B_2},\R) \ar@{-->}[rr] && \cA((U_2)_{B_2},\R)
    }
  \]
  Then boundedness of \eqref{eq:Biso4} is easily shown. 
  In the Roumieu cases $\cA \in \{\BrM,\SrLM\}$ we use the fact that any bounded 
  subset in $\cA(U_1 \times U_2,F)$ is contained and bounded in some step of the inductive limit describing $\cA(U_1 \times U_2,F)$ 
  and hence its image under \eqref{eq:Biso4} is contained and bounded in the corresponding step of the inductive limit 
  describing $\cA(U_2,F)$.

  Conversely, we need to show that
  \begin{equation} \label{eq:Biso5}
    \cA(U_1, \cA(U_2,F)) \ni g \mapsto  g^\wedge(x,y) \in F
  \end{equation}
  is bounded for all $(x,y) \in U_1 \times U_2$. But the mapping \eqref{eq:Biso5} is just the composite $\ev_y \o \ev_x$ and 
  thus bounded.

  Indeed, for any convenient vector spaces $E$, $F$, and $c^\infty$-open $U \subseteq E$, and each $x \in U$ 
  the evaluation mapping $\ev_x : \cA(U,F) \to F$ is bounded, since $\ell \o \ev_x$ is continuous for all $\ell \in F^*$,  
  by Subsection \ref{ssec:convenientstructure}. Alternatively, the $C^\infty$ exponential law yields 
  boundedness of $\ev : \cA(U,F) \times U \to F$ as follows: the mapping associated via the exponential law is 
  the inclusion $\cA(U,F) \to C^\infty(U,F)$ which obviously is smooth.
\end{proof}

\begin{remark}\label{rem-7.3}
	If $E_i$, $F$ are Banach spaces and $\cA \in \{\cB,\cS,\BbM,\SbLM\}$ then we even get topological 
	isomorphisms
	\begin{equation} \label{eq:Biso1}
    \cA(U_1 \times U_2,F) \ni  f \mapsto f^\vee \in  \cA(U_1,\cA(U_2,F))  
  	\end{equation}	
  	provided that we equip $\cA(U_1,\cA(U_2,F))$ ($= \cA_b(U_1,\cA_b(U_2,F))$ by Lemma \ref{AvsAb}) with the 
  	Fr\'echet topology generated by the basis of neighborhoods of zero
  	\begin{align} \label{eq:Biso2}
    	\{g : \set(g)  \subseteq \cV_\ell\}    
  	\end{align} 
  	where $\set \in \fS_\cA$ 
  	and $\{\cV_\ell\}$ is a basis of neighborhoods of zero in $\cA(U_2,F)$. 
  	It is easy to see that the mapping \eqref{eq:Biso1} is then continuous, and thus the statement follows from 
  	the open mapping theorem. 
\end{remark}

\subsection{The exponential law for smooth functions with compact support}  \label{ssec:Dexp}

	For locally convex spaces $F$ let
	\begin{align*}
		\cD(\R^\ell,F)&:=\varinjlim_{K\in\sK(\R^\ell)} C^\oo_K(\R^\ell,F),\\
		\intertext{where}
		C^\oo_K(\R^\ell,F) &:=\{f\in C^\oo(\R^\ell,F):f(x)=0\;\forall x\notin K\},
	\end{align*}
	supplied with the locally convex injective limit topology for the former
	and the subspace topology induced from the topology of uniform convergence
	in each derivative separately on $C^\oo(\R^\ell,F)$ for the latter space.
  Note that on $C^\oo_K(\R^\ell,F)$ this coincides with the 
  topology induced from $\cBb(\R^\ell,F)$ mentioned in the Remark
  \ref{rem-7.3}. Thus a subset $\mathcal F\subs C^\oo_K(\R^\ell,F)$
  is bounded therein if and only if for each multi-index $\al$
  the set $\{f^{(\al)}(x) : x \in \R^\ell,\; f \in \cF\}$ is bounded in $F$.
  By Lemma \ref{AvsAb}, this is in turn equivalent to the boundedness in $\cB(\R^\ell,F)$.
  The injective limit is a strict inductive limit hence regular
	since $C^\oo_{K'}(\R^\ell,F)$ is a closed topological subspace of $C^\oo_K(\R^\ell,F)$
	for every $K'\subs K$.

	If, in addition, $M=(M_k)$ is a weight sequence, then let
	\begin{align*}
		\DM(\R^\ell,F)&:=\varinjlim_{K\in\sK(\R^\ell)} \CM_K(\R^\ell,F),\\
		\intertext{where}
		\CM_K(\R^\ell,F) &:=\{f\in \CM(\R^\ell,F):f(x)=0\;\forall x\notin K\},
	\end{align*}
	supplied with the locally convex injective limit topology for the former space
	and the subspace topology induced from $\BM(\R^\ell,F)$ for the latter space.
	Again the injective limit is a strict inductive limit hence regular
	since $\CM_{K'}(\R^\ell,F)$ is a closed topological subspace of $\CM_K(\R^\ell,F)$
	for every $K'\subs K$.

\begin{theorem} \label{thm:Dexp}
	Let $M=(M_k)$ be a weakly log-convex weight sequences with moderate growth.
  We have bornological isomorphisms
	\begin{align*}
		\cD(\R^\ell\x \R^m,\R^n)&\cong \cD(\R^\ell,\cD(\R^m,\R^n))\\
		\DM(\R^\ell\x \R^m,\R^n)&\cong \DM(\R^\ell,\DM(\R^m,\R^n))
	\end{align*}
\end{theorem}

\begin{demo}{Proof}
	Let us first consider the case $\cD$.
	The bounded subsets $\mathcal F\subs\cD(\R^\ell,F)$ in this regular inductive limit are
	exactly those sets for which there exists a compact subset $K\subs\R^\ell$
	such that $\mathcal F$ is contained and bounded in $C^\oo_K(\R^\ell,F)\subs \cB(\R^\ell,F)$.

	Thus a subset $\mathcal F\subs\cD(\R^\ell,\cD(\R^m,\R^n))$ is bounded if and only if 
	there exists a compact $K\subs \R^\ell$
	such that $f(x)=0$ for all $f\in\mathcal F$ and all $x\notin K$
	and $\mathcal F$ is bounded in $\cB(\R^\ell,\cD(\R^m,\R^n))$. Boundedness in 
	$\cB(\R^\ell,\cD(\R^m,\R^n))$ means, that for every
	multi-index $\al$ there exists a compact set $K^\al\subs\R^m$
	such that $\{f^{(\al)}(x):x\in\R^\ell,f\in\mathcal F\}$ is contained in
	$C^\oo_{K^\al}(\R^m,\R^n)$ and is bounded in $\cB(\R^m,\R^n)$.
	Since $f^{(\al)}(x)(y)=0$ provided $f(x)(y)=0$ for all $x$, it is enough to consider 
	$\al=0$. Thus the boundedness of $\mathcal F$ is equivalent to the existence
	of the compact set $K\x K^0$ such that $f^\wedge(x,y)=0$ 
	for all $(x,y)\notin K\x K^0$ and the boundedness in $\cB(\R^\ell,\cB(\R^m,\R^n))$.
	By Theorem \ref{thm-7.2} this is equivalent
	to the boundedness of $\{f^\wedge:f\in\mathcal F\}$ in $\cB(\R^\ell\x\R^m,\R^n)$
	and thus to that of $\{f^\wedge:f\in\mathcal F\}$ in $\cD(\R^\ell\x \R^m,\R^n)$.

	We consider now the case $\DM$:
	A subset $\mathcal F\subs\DM(\R^\ell,F)$ is bounded if and only if 
  there exists a compact $K\subs \R^\ell$
  such that $f(x)=0$ for all $f\in\mathcal F$ and all $x\notin K$
  and $\mathcal F$ is bounded in $\BM(\R^\ell,F)$.
	For $F=\mathbb R^n$ or $F=\DM(\mathbb R^m,\mathbb R^n)$,  by Lemma \ref{AvsAb}, the set 
  $\mathcal F$ is bounded in $\BM(\R^\ell,F)$ if and only it is bounded in $\BMb(\R^\ell,F)$;
	here we use that $\DrM(\R^m,\R^n)$ is a Silva space 
  and hence satisfies the assumptions of Lemma \ref{AvsAb}. 
  Boundedness of $\cF$ in $\BM(\R^\ell,\DM(\R^m,\R^n))$ means boundedness of  
	\[
    \set_{\rh} := \Big\{ \frac{f^{(\al)}(x)}{\rh^{|\al|} |\al|! M_{|\al|}} : \al \in \N^\ell,\; x \in \R^\ell,\; f \in \cF\Big\}
  \]
  in $\DM(\R^m,\R^n)$ for some $\rh>0$ if $[~]=\{~\}$, or for all $\rh>0$ if $[~]=(~)$.
	So there exists a compact subset $K' = K'_\rh\subseteq \R^m$ such that 
  $\set_\rh$ is contained in $C^{[M]}_{K'}(\R^m,\R^n)$ and is bounded in $\BM(\R^m,\R^n)$.
  Thus $\cF$ is bounded in $\DM(\R^\ell,\DM(\R^m,\R^n))$ if and only if there exists a 
	compact set $K\x K'$ such that $f^\wedge(x,y)=0$ for all $f\in\cF$ and
  for all $(x,y)\notin K\x K'$, and $\cF$ is bounded in $\BM(\R^\ell,\BM(\R^m,\R^n))$.  
  By Theorem \ref{thm-7.2} the latter is equivalent
  to the boundedness of $\{f^\wedge:f\in\mathcal F\}$ in $\BM(\R^\ell\x\R^m,\R^n)$.
	Thus $\cF$ is bounded in $\DM(\R^\ell,\DM(\R^m,\R^n))$	if and only if 
  $\{f^\wedge:f\in\mathcal F\}$ is bounded  in $\DM(\R^\ell\x \R^m,\R^n)$.
\qed\end{demo}

\section{The \texorpdfstring{$\Sp$}{Sp} and \texorpdfstring{$\SpM$}{SpM} exponential law} \label{sec:Wexp}

In this section we prove the exponential law for $\Sp$ and for $\SpM$. 

\subsection{In finite dimensions}

\paragraph{\bf Smooth functions with globally $p$-integrable derivatives}

For $p \in [1,\infty]$ consider  
\begin{align*}
  \Sp(\R^m,\R) &= \Sp(\R^m) = \bigcap_{k \in \N} W^{k,p}(\R^m) \\ 
    &= \Big\{f \in C^\infty(\R^m) : \|f^{(\al)}\|_{L^p(\R^m)} < \infty \text{ for all } \al \in \N^m\Big\}  
\end{align*}
with its natural Fr\'echet topology, and set
\[
  \Sp(\R^m,\R^n) := (\Sp(\R^m,\R))^n.  
\]
These classes where denoted by $\cD_{L^p}$ in \cite[p.~199]{Schwartz66}.
The most important case is
$W^{\infty,2}(\R^m) = \Hoo(\R^m)$.
Note that $W^{\infty,\infty}(\R^m) =\cB(\R^m)$, so henceforth we restrict ourselves to the case $p \in [1,\infty)$.

\paragraph{\bf Sobolev--Denjoy--Carleman classes}

Let $M=(M_k)$ be a weight sequence, let $p \in[1,\infty)$, and let $\rh>0$.
Consider the Banach space 
\[
  W^{M,p}_{\rh}(\R^m) := \big\{f \in C^\infty(\R^m) : \|f\|_{\R^m,\rh}^{M,p} < \infty\big\},
\]
where
\[
  \|f\|_{\R^m,\rh}^{M,p} := \sup_{\al \in \N^m} \frac{\|f^{(\al)}\|_{L^p(\R^m)}}{\rh^{|\al|}\, |\al|!\, M_{|\al|}}.
\]
We define the Fr\'echet space 
\[
  \SpbM(\R^m) := \proj_{n \in \N} W^{M,p}_{\frac{1}{n}}(\R^m) 
\]
and  
\[
  \SprM(\R^m) := \ind_{n \in \N} W^{M,p}_{n}(\R^m)
\]
which is a compactly regular (LB)-space and thus {($c^\oo$-)}complete, webbed, and (ultra-)bornological, 
see \cite[Lemma 4.9]{KrieglMichorRainer14a}, and set
\[
   \SpM(\R^m,\R^n) := (\SpM(\R^m))^n.
\]

\begin{proposition}[{\cite[Prop.~5.1]{KrieglMichorRainer14a}}] \label{prop:incl2}
  Let $M=(M_k)$ and $L=(L_k)$ be weight sequences, where $L_k\ge 1$ for all $k$.
  We have the following inclusions, where we omit the source $\R^m$ and the target $\R^n$, i.e., we write $\cA$ instead of $\cA(\R^m,\R^n)$.  
  Let $1 \le p<q <\infty$.
  For the inclusions marked by $*$ we assume that $M=(M_k)$ is derivation closed. 
    \[
      \xymatrix{
        \cD~ \ar@{{ >}->}[r] & \cS~ \ar@{{ >}->}[r] & \Sp~ \ar@{{ >}->}[r] & W^{\infty,q}~ \ar@{{ >}->}[r] 
        & \cB~ \ar@{{ >}->}[r] & C^\infty \\
        \DrM~ \ar@{{ >}->}[r] \ar@{{ >}->}[u] & \SrLM~ \ar@{{ >}->}[r] \ar@{{ >}->}[u] & \SprM~ \ar@{{ >}->}[r]^{*} \ar@{{ >}->}[u] 
        & W^{\rM,q}~ \ar@{{ >}->}[r]^{*} \ar@{{ >}->}[u] 
        & \BrM~ \ar@{{ >}->}[u] \ar@{{ >}->}[r] & \CrM~ \ar@{{ >}->}[u]\\
        \DbM~ \ar@{{ >}->}[r] \ar@{{ >}->}[u] &\SbLM~ \ar@{{ >}->}[r] \ar@{{ >}->}[u] & \SpbM~ \ar@{{ >}->}[r]^{*} \ar@{{ >}->}[u] 
        & W^{\bM,q}~ \ar@{{ >}->}[r]^{*} \ar@{{ >}->}[u] 
        & \BbM~ \ar@{{ >}->}[u] \ar@{{ >}->}[r] & \CbM~ \ar@{{ >}->}[u]
      }
    \]
   All inclusions are continuous. 
   If the target is $\R$ (or $\C$) then all spaces are algebras, provided that $M=(M_k)$ is weakly log-convex,
   and each space in 
   \[
      \xymatrix{
        \cD(\R^m)~ \ar@{{ >}->}[r] & \cS(\R^m)~ \ar@{{ >}->}[r] & \Sp(\R^m)~ \ar@{{ >}->}[r] & W^{\infty,q}(\R^m)~ \ar@{{ >}->}[r] 
        & \cB(\R^m) 
      }
    \]
    is a $\cB(\R^m)$-module, and thus an ideal in each space on its right, likewise 
    each space in 
   \[
      \xymatrix{
        \DM(\R^m) \ar@{{ >}->}[r] & \SLM(\R^m) \ar@{{ >}->}[r] & \SpM(\R^m) \ar@{{ >}->}[r] & W^{[M],q}(\R^m) \ar@{{ >}->}[r] 
        & \BM(\R^m) 
      }
    \]
    is a $\BM(\R^m)$-module, and thus an ideal in each space on its right.  
\end{proposition}

\begin{remark*}
  The fact that 
  $\cD$ is dense in $\Sp$ (but not in $\cB$) and the Sobolev inequality imply that each element of $\Sp$ must tend to $0$ at infinity 
  together with all its iterated partial derivatives. 
\end{remark*}

\begin{lemma} \label{lem:H1}
  Let $f \in C^1(\R^{1+n}) \cap W^{1,p}(\R^{1+n})$ and let $x_0 \in \R$. 
  Then $f^\vee(x_0) = f(x_0, ~) \in C^1(\R^n) \cap L^p(\R^n)$ and  
  \begin{equation} \label{eq:vee}
    \|f^\vee(x_0)\|_{L^p(\R^n)} \le C \, \|f\|_{W^{1,p}(\R^{1+n})},
  \end{equation}
  for a universal constant $C$.
\end{lemma}

In particular, the set $\{f^\vee(x) : x \in \R\}$ is bounded in $L^p(\R^n)$.

\begin{proof}
  Choose a decreasing $C^\infty$-function $\ph : \R \to \R$ satisfying $\ph|_{\{x\le 0\}} = 1$ and $\ph|_{\{x\ge 1\}} = 0$. 
  Let $B(x_0,r) : = \{(x,y) \in \R^2 : (x-x_0)^2 + y_1^2 + \cdots +y_n^2 < r^2\}$ be the open ball of radius $r$ centered at $(x_0,0) \in \R^{1+n}$  
  and let $B_+(x_0,r) := B(x_0,r) \cap \{x>x_0\}$ be its right half. 
  We define 
  \[
    \ps(x,y) := \ph\Big(\sqrt{(x-x_0)^2 + y_1^2 + \cdots +y_n^2} - r\Big).
  \]
  Then $\ps= 1$ on $B(x_0,r)$ and $\ps = 0$ outside of $B(x_0,r+1)$. 
  Since $|f|^p$ is locally Lipschitz and since  
  $\p_x (|f|^p) = p\, |f|^{p-1} (\on{sgn} f) (\p_x f)$ a.e.\ (see, e.g., \cite[Thm~2.1.11]{Ziemer89} or \cite[Thm~6.17]{LiebLoss01}), the fundamental theorem of calculus implies
  \begin{align*}
    \int_{B(x_0,r)\cap \{x=x_0\}} |f|^p\, dy &\le \int_{\{x=x_0\}} \ps|f|^p\, dy  \\  
    &= - \int_{B_+(x_0,r+1)} \p_x \big(\ps|f|^p\big)\, d(x,y) \\ 
    &= - \int_{B_+(x_0,r+1)} (\p_x \ps) |f|^p + p\, \ps |f|^{p-1} (\on{sgn} f) (\p_x f)\, d(x,y) \\ 
    &\le \int_{B_+(x_0,r+1)} |\p_x \ps| |f|^p + p\, |\ps| |f|^{p-1} |\p_x f|\, d(x,y) \\
    &\le C \int_{B_+(x_0,r+1)} |f|^p + |\p_x f|^p\, d(x,y)
  \end{align*} 
  for some constant only depending on $\ph$,
  using $|\ps| \le 1$,
  \begin{align*}
    |\p_x \ps| &=  \frac{|x-x_0| |\ph'\big(\sqrt{(x-x_0)^2 + y_1^2 + \cdots +y_n^2} - r\big)|}{\sqrt{(x-x_0)^2 + y_1^2 + \cdots +y_n^2}} 
    \le \|\ph'\|_{L^\infty(\R)} 
  \end{align*}
  and $p |f|^{p-1} |\p_x f| \le (p-1) |f|^p + |\p_x f|^p$, by Young's inequality.  
  Letting $r \to +\infty$ implies the statement.  
\end{proof}

\begin{proposition} \label{prop:H2}
    If $f \in \Sp(\R^{\ell+n})$ and $x_0 \in \R^\ell$, then 
    $f^\vee(x_0) \in \Sp(\R^n)$ 
    and 
    \begin{equation} \label{eq:dvee}
    \|(f^\vee(x_0))^{(\al)}\|_{L^p(\R^n)} \le C \, \|f\|_{W^{|\al|+\ell,p}(\R^{\ell+n})},\quad \al \in \N^n,
    \end{equation}
    for a universal constant $C$.  
\end{proposition}

\begin{proof}
  For $\ell = 1$ this follows from applying Lemma \ref{lem:H1} to $\p_y^\al f(x,y)$. 
  The general statement follows by induction on $\ell$.
\end{proof}

\subsection{Vector-valued functions of class $\Sp$ and $\SpM$}

Let $M=(M_k)$ be a weight sequence.
For a locally convex space $F$ we define 
\begin{empheq}[box=\widefbox]{align*}
  \Sp(\R^m,F) &:= 
  \Big\{f\in C^\infty(\R^m,F):
  \forall \al ~\forall s: \|s \o f^{(\al)}\|_{L^p(\R^m)}< \infty\Big\},\\
  \SpbM(\R^m,F) &:= 
  \Big\{f\in C^\infty(\R^m,F):
  \forall s ~\forall \si:
  \sup_{\al \in \N^m} \frac{\|s \o f^{(\al)}\|_{L^p(\R^m)}}{\si^{|\al|} |\al|! M_{|\al|}} < \infty\Big\},\\
  \SprM(\R^m,F) &:= 
  \Big\{f\in C^\infty(\R^m,F):
  \forall s ~\exists \si:
  \sup_{\al \in \N^m} \frac{\|s \o f^{(\al)}\|_{L^p(\R^m)}}{\si^{|\al|} |\al|! M_{|\al|}} < \infty\Big\},
\end{empheq}
where $\al \in \N^m$, $s \in \SN(F)$, $\si>0$, and 
$\SN(F)$ is the collection of all continuous seminorms on $F$. 


We shall need a projective description for $\SprM(\R^m,F)$.
For a $C^\infty$-mapping $f : \R^m \to F$ into a locally convex space $F$, a positive sequence $(r_k)$, $p \in [1,\infty)$, 
and $s\in \SN(F)$ consider 
\[
  \set^{M,p}_{s,(r_k)}(f) := \Big\{r_{|\al|} \frac{\|s \o f^{(\al)}\|_{L^p(\R^m)}}{|\al|!\, M_{|\al|}}  : \al\in \mathbb N^m\Big\},
\]
and define 
\[
  \fS^R_{\SprM} := \{\set^{M,p}_{s,(r_k)} : s\in \SN(F), (r_k) \in \sR\},
\]
where $\sR$ and $\sR'$ were defined in Subsection \ref{sec:proj}. 

\begin{lemma} \label{lem:AEW}
For a $C^\infty$-mapping $f : \R^m \to F$ into a locally convex space $F$
the following are equivalent.
\begin{enumerate}
\item $f$ is $\SprM$.
\item For each $\set \in \fS_{\SprM}^R$, 
the set $\set(f)$ is bounded in $\R$.
\item For $s \in \SN(F)$ and each sequence $(r_k) \in \sR'$ 
there exists 
$\de>0$ such that the set $\set^{M,p}_{s,(r_k\, \de^k)}(f)$
is bounded in $\R$. 
\end{enumerate}
\end{lemma}

\begin{proof}
  Set
  \begin{align*}
    a_{s,\al}:=\frac{\|s \o f^{(\al)}\|_{L^p(\R^m)}}{|\al|!\, M_{|\al|}!},
  \end{align*}
  and apply Lemma \ref{sec:proj}. 
\end{proof}

\begin{lemma} \label{lem:WR}
  We have $f \in \SprM(\R^\ell,\SprM(\R^m))$ if and only if  
  \begin{equation} \label{eq:veeR1}
    \exists \si, \ta>0 : 
    ~\sup_{\al \in \N^\ell,\be \in \N^m} 
    \frac{(\int_{\R^\ell} \|\p_2^{\be}[f^{(\al)}(x)]\|_{L^p(\R^m)}^p dx)^{\frac 1 p}}{\ta^{|\be|}\si^{|\al|}\, |\be|! |\al|!\, M_{|\be|} M_{|\al|}} < \infty.
  \end{equation}
\end{lemma}

\begin{proof}
  By definition $f \in \SprM(\R^\ell,\SprM(\R^m))$ if and only if
  \begin{equation} \label{eq:veeR2}
    \forall s \in \SN(\SprM(\R^m)) ~\exists \si >0 : 
    ~\sup_{\al \in \N^\ell} \frac{\|s \o f^{(\al)}\|_{L^p(\R^\ell)}}{\si^{|\al|}\, |\al|!\, M_{|\al|}}  < \infty.
  \end{equation}
  Now, if we denote by $i_\ta : W^{M,p}_{\ta}(\R^m) \hookrightarrow \SprM(\R^m)$ the canonical inclusion and $s$ is 
  a seminorm on $\SprM(\R^m)$, then
  \begin{align*}
    s \in \SN(\SprM(\R^m)) &\Longleftrightarrow ~\forall \ta>0 : s \o i_\ta \in \SN(W^{M,p}_{\ta}(\R^m)) \\
    &\Longleftrightarrow ~\forall \ta>0 ~\exists C>0 : s \o i_\ta \le C \|~\|^{M,p}_{\R^m,\ta}. 
  \end{align*}
  Thus \eqref{eq:veeR1} implies \eqref{eq:veeR2}.

  Let us prove the converse. 
  By Lemma \ref{lem:AEW}, \eqref{eq:veeR2} is equivalent to 
  \begin{equation} \label{eq:SprM1}
    \forall s \in \SN(\SprM(\R^m)) ~\forall (r_k) \in \sR : 
    ~\sup_{\al \in \N^\ell} r_{|\al|}  \frac{\|s \o f^{(\al)}\|_{L^p(\R^\ell)}}{|\al|!\, M_{|\al|}}  < \infty.
  \end{equation}
  For $(t_k) \in \sR$ and $g \in C^\infty(\R^m)$ set 
  \[
    \|g\|^{M,p}_{\R^m,(t_k)} := \sup_{\be \in \N^m} t_{|\be|} \frac{\|g^{(\be)}\|_{L^p(\R^m)}}{|\be|!\, M_{|\be|}}.
  \]
  If $g \in \SprM(\R^m)$, then there exist $\si,C>0$ so that 
  \[
    \|g\|^{M,p}_{\R^m,(t_k)} = \sup_{\be \in \N^m} t_{|\be|} \si^{|\be|} \frac{\|g^{(\be)}\|_{L^p(\R^m)}}{\si^{|\be|} \,|\be|!\, M_{|\be|}} 
    \le C \|g\|^{M,p}_{\R^m,\si};
  \]
  cf.\ the proof of Lemma \ref{sec:proj}, \thetag{1} $\Rightarrow$ \thetag{2}. That is $\|~\|^{M,p}_{\R^m,(t_k)} \in \SN(\SprM(\R^m))$ 
  for all $(t_k) \in \sR$. 
  Thus \eqref{eq:SprM1} implies
  \begin{align*}
    \forall (t_k),(r_k) \in \sR  : 
    ~\sup_{\al \in \N^\ell,\be \in \N^m} 
    t_{|\al|} r_{|\be|} \frac{(\int_{\R^\ell} \int_{\R^m} |\p_2^\be \p_1^\al f^\wedge(x,y)|^p\, dy dx)^{\frac 1 p}}
    {|\be|! |\al|!\, M_{|\be|} M_{|\al|}} < \infty.
  \end{align*}
  In particular, for $t_k=r_k$ and assuming $r_k r_j \ge r_{k+j}$ for all $k,j$, we have 
  \begin{align*}
    \forall (r_k) \in \sR' : 
    ~\sup_{(\al,\be) \in \N^\ell \times \N^m} 
    r_{|\al|+|\be|} \frac{(\int_{\R^\ell} \int_{\R^m} |\p_2^\be \p_1^\al f^\wedge(x,y)|^p\, dy dx)^{\frac 1 p}}
    {|\be|! |\al|!\, M_{|\be|} M_{|\al|}} < \infty.
  \end{align*}
  Applying Lemma \ref{sec:proj} to
  \[
    a_{\al,\be} := \frac{\Big(\int_{\R^\ell} \int_{\R^m} |\p_2^\be \p_1^\al f^\wedge(x,y)|^p\, dy dx\Big)^{1/p}}
    {|\be|! |\al|!\, M_{|\be|} M_{|\al|}}, 
  \]
  we may conclude that 
  \begin{align*}
    \exists \si>0 : 
    ~\sup_{(\al,\be) \in \N^\ell \times \N^m} 
     \frac{(\int_{\R^\ell} \int_{\R^m} |\p_2^\be \p_1^\al f^\wedge(x,y)|^p\, dy dx)^{\frac 1 p}}
    {\si^{|\al|+|\be|}|\be|! |\al|!\, M_{|\be|} M_{|\al|}} < \infty,
  \end{align*}
  that is \eqref{eq:veeR1}.
\end{proof}

Now we are ready to prove the exponential law.

\begin{theorem} \label{thm:Wexp}
  Let $M=(M_k)$ be a weakly log-convex weight sequence with moderate growth.
  We have
  \begin{align*}
    \Sp(\R^\ell \times \R^m,\R^n) &= \Sp(\R^\ell,\Sp(\R^m,\R^n)), \\
    \SpM(\R^\ell \times \R^m,\R^n) &= \SpM(\R^\ell,\SpM(\R^m,\R^n)).
  \end{align*}
\end{theorem}

\begin{remark*}
  In the $\SpM$-exponential law the inclusion \thetag{$\supseteq$} holds 
  without $M=(M_k)$ having moderate growth, 
  the inclusion \thetag{$\subseteq$}  
  without $M=(M_k)$ being weakly log-convex.
\end{remark*}

\begin{proof}
  Let $\cA \in \{\Sp,\SpM\}$. We may assume without loss of generality that $n=1$.

  \paragraph{{\bf Direction} ($\Rightarrow$)} 
  Let $f \in \cA(\R^\ell \times \R^m)$.  
  By Proposition \ref{prop:H2} and as $M=(M_k)$ has moderate growth and is thus derivation closed, 
  $f^\vee$ takes values in $\cA(\R^m)$. 
  Moreover, the mapping $f^\vee:\R^\ell\to \cA(\R^m)$ is $C^\oo$ with $d^jf^\vee=(\d_1^jf)^\vee$; this can be proved 
  as follows.

  Since $\cA(\R^m)$ is a convenient vector space, by \cite[5.20]{KM97}
  it is  enough to show that the iterated unidirectional
  derivatives $d^j_vf^\vee(x)$ exist, equal $\p_1^jf(x,~)(v^j)$, 
  and are separately bounded for $x$, resp.\  $v$, in compact 
  subsets.
  For $j=1$ and fixed $x$, $v$, and $y$ consider the smooth curve
  $c:t\mapsto f(x+tv,y)$. By the fundamental theorem
  \begin{align*}
    \frac{f^\vee(x+tv)-f^\vee(x)}{t}(y)&-(\d_1f)^\vee(x)(y)(v) 
    = \frac{c(t)-c(0)}t-c'(0) \\
    &= t\int_0^1 s \int_0^1 c''(tsr)\,dr\,ds \\
    &= t\int_0^1 s\int_0^1 \d_1^2f(x+tsrv,y)(v,v)\,dr \,ds.
  \end{align*}
  By Lemma \ref{lem:H1} and as $M=(M_k)$ is derivation closed, 
  $(\d_1^2f)^\vee(K)(B,B)$ is bounded in $\cA(\R^m)$ for each compact subset
  $K\subs \R^\ell$ and the closed unit ball $B \subseteq \R^\ell$, 
  and so this expression is Mackey convergent  
  to 0 in $\cA(\R^m)$ as $t\to 0$.
  Thus $d_vf^\vee(x)$ exists and equals $\d_1f(x,~)(v)$.
  
  Now we proceed by induction, applying the same arguments as before to
  $(d^j_vf^\vee)^\wedge : (x,y)\mapsto \p_1^jf(x,y)(v^j)$ instead of $f$.
  Again
  $(\p_1^2(d^j_vf^\vee)^\wedge)^\vee(K)(B,B) = (\p_1^{j+2}f)^\vee(K)(B,B,v,\dots,v)$
  is bounded, and also the separated boundedness of $d^j_vf^\vee(x)$ follows.
  So the claim is proved.

  Next we show that 
  \begin{equation} \label{eq:Hvee}
    \text{$f^\vee : \R^\ell \to \cA(\R^m)$ is $\cA$}.  
  \end{equation}
  Note that by Fubini's theorem 
  \begin{align} \label{eq:Fubini}
  \begin{split}
    \int_{\R^\ell} \|\p_2^\be[(f^\vee)^{(\al)}(x)]\|_{L^p(\R^m)}^p dx 
    &= \int_{\R^\ell} \int_{\R^m} |\p_2^\be \p_1^\al f(x,y)|^p\, dy dx 
    \\
    &= \|f^{(\al,\be)}\|_{L^p(\R^\ell \x \R^m)}^p.
  \end{split} 
  \end{align}

  \paragraph{\bf Case $\cA =\Sp$}
  \eqref{eq:Hvee} is equivalent to 
  \begin{gather}  \label{eq:H2}
  ~\forall \al \in \N^\ell~\forall s \in \SN(\Sp(\R^m)) : \int_{\R^\ell} [s((f^\vee)^{(\al)}(x))]^p dx < \infty \\
  \Updownarrow \nonumber \\
  ~\forall \al\in \N^\ell,\be \in \N^m : 
  ~\int_{\R^\ell} \|\p_2^\be[(f^\vee)^{(\al)}(x)]\|_{L^p(\R^m)}^p dx =  \|f^{(\al,\be)}\|_{L^p(\R^\ell \x \R^m)}^p
  < \infty 
   \label{eq:H2'}
  \end{gather}
  which is true, since $f \in \Sp(\R^\ell \times \R^m)$, by assumption.  

  \paragraph{\bf Case $\cA =\SpbM$}
  \eqref{eq:Hvee} is equivalent to 
  \begin{gather}  \label{eq:H2M}
  ~\forall \si >0~\forall s \in \SN(\SpbM(\R^m)) : 
  ~\sup_{\al \in \N^\ell} \frac{\|s \o (f^\vee)^{(\al)}\|_{L^p(\R^\ell)}}{\si^{|\al|}\, |\al|!\, M_{|\al|}}  < \infty 
  \\
  \Updownarrow \nonumber \\
  \forall \si,\ta>0 : 
  ~\sup_{\al \in \N^\ell,\be \in \N^m} 
  \frac{(\int_{\R^\ell} \|\p_2^\be[(f^\vee)^{(\al)}(x)]\|_{L^p(\R^m)}^p dx)^{\frac 1 p}}
  {\ta^{|\be|}\si^{|\al|}\, |\be|! |\al|!\, M_{|\be|} M_{|\al|}}  
  \label{eq:H2M'} \\
  \hspace{5cm} =
  \sup_{\al \in \N^\ell,\be \in \N^m} 
  \frac{\|f^{(\al,\be)}\|_{L^p(\R^\ell \x \R^m)}}{\ta^{|\be|}\si^{|\al|}\, |\be|! |\al|!\, M_{|\be|} M_{|\al|}} < \infty 
   \nonumber 
  \end{gather}
  which is true, as $M=(M_k)$ has moderate growth (see \thetag{\ref{ssec:ws}.\ref{eq:mg0}} or \thetag{\ref{thm:Bexp}.\ref{eq:mg}}), 
  since $f \in \SpbM(\R^\ell \times \R^m)$, by assumption.

  \paragraph{\bf Case $\cA =\SprM$}  
  The assumption $f \in \SprM(\R^\ell \times \R^m)$ implies, 
  as $M=(M_k)$ has moderate growth and by \eqref{eq:Fubini}, 
  that 
  \begin{gather} \label{eq:WR}
    ~\exists \si,\ta>0 : 
  ~\sup_{\al \in \N^\ell,\be \in \N^m} 
  \frac{(\int_{\R^\ell} \|\p_2^\be[(f^\vee)^{(\al)}(x)]\|_{L^p(\R^m)}^p dx)^{\frac 1 p}}{\ta^{|\be|}\si^{|\al|}\, |\be|! |\al|!\, M_{|\be|} M_{|\al|}}  
   \\
  \hspace{5cm} =
  \sup_{\al \in \N^\ell,\be \in \N^m} 
  \frac{\|f^{(\al,\be)}\|_{L^p(\R^\ell \x \R^m)}}{\ta^{|\be|}\si^{|\al|}\, |\be|! |\al|!\, M_{|\be|} M_{|\al|}} < \infty. 
  \nonumber
  \end{gather}
  By Lemma \ref{lem:WR}, we may conclude \eqref{eq:Hvee}.

  \paragraph{{\bf Direction} ($\Leftarrow$)}
  Let $f^\vee : \R^\ell \to \cA(\R^m)$ be $\cA$. 
  Then $f^\vee: \R^\ell \to \cA(\R^m)\to C^\oo(\R^m)$
  is $C^\oo$, since the latter inclusion is bounded, 
  by the general Sobolev inequalities.

  \paragraph{\bf Case $\cA =\Sp$}
  That $f^\vee : \R^\ell \to \Sp(\R^m)$ is $\Sp$ is equivalent to \eqref{eq:H2} 
  which in turn yields that $f \in \Sp(\R^\ell \times \R^m)$.

  \paragraph{\bf Case $\cA =\SpbM$}
  That $f^\vee : \R^\ell \to \SpbM(\R^m)$ is $\SpbM$ is equivalent to \eqref{eq:H2M},  
  which implies $f \in \SpbM(\R^\ell \times \R^m)$, as $M=(M_k)$ is weakly log-convex 
  (see \thetag{\ref{ssec:ws}.\ref{eq:alg}} or \thetag{\ref{thm:Bexp}.\ref{eq:wlc}}).

  \paragraph{\bf Case $\cA =\SprM$}
  By Lemma \ref{lem:WR}, $f^\vee \in  \SprM(\R^\ell,\SprM(\R^m))$ if and only if \eqref{eq:WR} holds, 
  and \eqref{eq:WR} implies $f \in \SprM(\R^\ell \times \R^m)$, as $M=(M_k)$ is weakly log-convex. 
\end{proof}

\subsection{Topology on $\Sp(\R^\ell,\Sp(\R^m,\R^n))$ and $\SpM(\R^\ell,\SpM(\R^m,\R^n))$} \label{ssec:top}

On $\Sp(\R^\ell,\Sp(\R^m,\R^n))$ we consider  
the Fr\'echet topology generated by the following 
fundamental system of seminorms
  \begin{align} \label{eq:Wtop}
    g \mapsto \Big(\int_{\R^\ell} \|\p_2^\be (g^{(\al)}(x))\|_{L^p(\R^m)}^p \, dx\Big)^{\frac{1}{p}}, \quad \al \in \N^\ell, \be \in \N^m. 
  \end{align} 
Analogously, we consider on $\SpbM(\R^\ell,\SpbM(\R^m,\R^n))$   
the Fr\'echet topology generated by the fundamental system of seminorms 
  \begin{align} \label{eq:Wtop2}
    g \mapsto  \sup_{\al \in \N^\ell,\be \in \N^m} 
  \frac{(\int_{\R^\ell} \|\p_2^\be(g^{(\al)}(x))\|_{L^p(\R^m)}^p dx)^{\frac{1}{p}}}{\ta^{|\be|}\si^{|\al|}\, |\be|! |\al|!\, M_{|\be|} M_{|\al|}}, 
  \quad \si,\ta>0.
  \end{align} 
In view of Lemma \ref{lem:AEW} we consider on $\SprM(\R^\ell,\SprM(\R^m,\R^n))$   
the locally convex topology generated by the fundamental system of seminorms 
  \begin{align} \label{eq:Wtop3}
    g \mapsto  \sup_{\al \in \N^\ell,\be \in \N^m} 
  t_{|\be|} r_{|\al|}\,  
  \frac{(\int_{\R^\ell} \|\p_2^\be(g^{(\al)}(x))\|_{L^p(\R^m)}^p dx)^{\frac{1}{p}}}{|\be|! |\al|!\, M_{|\be|} M_{|\al|}}, 
  \quad (r_k), (t_k) \in \sR.
  \end{align}

\begin{theorem} \label{thm:Wiso}
  Let $M=(M_k)$ be a weakly log-convex weight sequence with moderate growth.
  We have bornological isomorphisms
  \begin{align*}
    \Sp(\R^\ell \times \R^m,\R^n) &\cong \Sp(\R^\ell,\Sp(\R^m,\R^n)), \\
    \SpM(\R^\ell \times \R^m,\R^n) &\cong \SpM(\R^\ell,\SpM(\R^m,\R^n)),
  \end{align*}
  where the topology on the right-hand side is the one introduced in Subsection \ref{ssec:top}; 
  for $\Sp$ and $\SpbM$ the isomorphisms are even topological.
\end{theorem}

\begin{proof}
  For $\cA \in \{\Sp,\SpbM\}$ the statement follows from the open mapping 
  theorem if we show that the mapping 
  \begin{equation*}
    \cA(\R^\ell \times \R^m,\R^n) \ni f \mapsto f^\vee \in  \cA(\R^\ell,\cA(\R^m,\R^n)),  
  \end{equation*}
  which is bijective by Theorem \ref{thm:Wexp},
  is continuous. But this follows from (\ref{thm:Wexp}.\ref{eq:H2'}) and from (\ref{thm:Wexp}.\ref{eq:H2M'}). 

  For $\cA = \SprM$ we argue as follows. A subset $B$ is bounded in $\SprM(\R^\ell \times \R^m,\R^n)$ if and only if 
  \begin{equation} \label{eq:WrMiso}
    \E \si>0 :
    \sup_{\substack{\al \in \N^\ell,\be \in \N^m\\f \in B}} 
  \frac{\|f^{(\al,\be)}\|_{L^p(\R^\ell \times \R^m)}}{\si^{|\al|+|\be|}\, |\be|! |\al|!\, M_{|\be|} M_{|\al|}} < \infty;  
  \end{equation}
  by the properties of $M=(M_k)$.
  Using Lemma \ref{sec:proj} twice we may conclude that \eqref{eq:WrMiso} is equivalent to 
  \begin{equation}
    \A (r_k), (t_k) \in \sR :
    \sup_{\substack{\al \in \N^\ell,\be \in \N^m\\ f \in B}} 
    t_{|\be|} r_{|\al|}\,  
    \frac{\|f^{(\al,\be)}\|_{L^p(\R^\ell \times \R^m)}}{|\be|! |\al|!\, M_{|\be|} M_{|\al|}} < \infty
  \end{equation}
  which means exactly that $B^\vee$ is bounded in $\SprM(\R^\ell,\SprM(\R^m,\R^n))$
\end{proof}

\subsection{Tensor product representations}

It is well-known (see \cite[p.~199]{Schwartz66}) that $\cD(\R^n)$ is dense in $\Sp(\R^n)$. 
In the next lemma we show that $\DM(\R^n)$ is dense in $\SpM(\R^n)$ provided that $M=(M_k)$ is non-quasianalytic.

\begin{lemma*} \label{lem:dense}
  Let $M=(M_k)$ be a weakly log-convex non-quasianalytic weight sequence.
  Then $\DM(\R^n)$ is dense in $\SpM(\R^n)$.
\end{lemma*}

\begin{proof}
  Let $\ph \in \DM(\R^n)$ be such that $\ph|_{\{|x|\le 1\}} = 1$ and set $\ph_k(x):=\ph(x/k)$ for $k\in \N_{\ge1}$. 
  By Proposition \ref{prop:incl2}, $\ph_k f \in \DM(\R^n) \subseteq \SpM(\R^n)$ and we have 
  \begin{align*}
     |\p^\al (f-\ph_k f)(x)|  &\le |f^{(\al)}(x)| + \sum_{\be \le \al} \binom{\al}{\be} |f^{(\be)}(x)| |\ph^{(\al-\be)}(x/k)|   
     \\
     &\le 2 \|\ph\|^M_{\R^n,\rh} \sum_{\be \le \al} \binom{\al}{\be} |f^{(\be)}(x)|  \rh^{|\al-\be|} |\al-\be|!\, M_{|\al-\be|}.  
  \end{align*} 
  Since $1-\ph_k(x)$ vanishes for $|x| \le k$, we conclude that  
  \begin{align*}
     \frac{\|\p^\al (f-\ph_k f)\|_{L^p(\R^n)}}{\rh^{|\al|} |\al|!\, M_{|\al|}} 
     &\le  2 \|\ph\|^M_{\R^n,\rh} \sum_{\be \le \al} \binom{\al}{\be} \|f^{(\be)}\|_{L^p(\{|x|>k\})}   
     \frac{\rh^{|\al-\be|} |\al-\be|!\, M_{|\al-\be|}
     }{\rh^{|\al|} |\al|!\, M_{|\al|}}
     \\
     &\le  2 \|\ph\|^M_{\R^n,\rh} \sum_{\be \le \al} \binom{\al}{\be}    
     \frac{\|f^{(\be)}\|_{L^p(\{|x|>k\})}}{\rh^{|\be|} |\be|!\, M_{|\be|}}
     \\
     &\le  2 n^{|\al|} \|\ph\|^M_{\R^n,\rh}     \max_{\be \le \al}
     \frac{\|f^{(\be)}\|_{L^p(\{|x|>k\})}}{\rh^{|\be|} |\be|!\, M_{|\be|}},
   \end{align*} 
   where we used weak log-convexity of $M=(M_k)$,
   and consequently,
   \begin{align*}
     \sup_\al \frac{\|\p^\al (f-\ph_k f)\|_{L^p(\R^n)}}{(n \rh)^{|\al|} |\al|!\, M_{|\al|}} 
     &\le  2 \|\ph\|^M_{\R^n,\rh}     \sup_{\be}
     \frac{\|f^{(\be)}\|_{L^p(\{|x|>k\})}}{\rh^{|\be|} |\be|!\, M_{|\be|}} \to 0 
   \end{align*}
   as $k \to \infty$. This implies the assertion.
\end{proof}

\begin{theorem*}
  Let $M=(M_k)$ be a weakly log-convex non-quasianalytic weight sequence with moderate growth.
  We have linear topological isomorphisms
  \begin{gather*}
    \Sp(\R^m \x \R^n) 
    \cong \Sp(\R^m) \,\widehat{\otimes}\, \Sp(\R^n)\\
    \SpM(\R^m \x \R^n) 
    \cong \SpM(\R^m) \,\widehat{\otimes}\, \SpM(\R^n)
  \end{gather*}
  where $\Sp(\R^m) \,\widehat{\otimes}\, \Sp(\R^n)$ (resp.\ $\SpM(\R^m) \,\widehat{\otimes}\, \SpM(\R^n)$) denotes the completion with respect to the 
  topology on $\Sp(\R^m) \otimes \Sp(\R^n)$ (resp.\ $\SpM(\R^m) \otimes \SpM(\R^n)$) induced by its inclusion in 
  $\Sp(\R^m \x\R^n)$ (resp.\ $\SpM(\R^m \x\R^n)$).
\end{theorem*}

\begin{proof}
  All inclusions in the diagram
  \[
      \xymatrix{
        \Sp(\R^m) \otimes \Sp(\R^n) \ar@{{ >}->}[r]  
        & \Sp(\R^m \x \R^n) \\
        \cD(\R^m)\otimes \cD(\R^n)  \ar@{{ >}->}[u] \ar@{{ >}->}[r]  & \cD(\R^m \x \R^n) \ar@{{ >}->}[u]
      }
    \]
    as well as in 
  \[
      \xymatrix{
        \SpM(\R^m) \otimes \SpM(\R^n) \ar@{{ >}->}[r]  & \SpM(\R^m \x \R^n) \\
        \DM(\R^m)\otimes \DM(\R^n)  \ar@{{ >}->}[u] \ar@{{ >}->}[r] & \DM(\R^m \x \R^n) \ar@{{ >}->}[u]
      }
    \]
    are dense, by the lemma. That $\DM(\R^m)\otimes \DM(\R^n)$ is dense in $\DM(\R^m \x \R^n)$ can be seen 
    as in the proof of \cite[Thm~2.1]{Komatsu77}. 
\end{proof}

\begin{problem*}
  Find an explicit description of the topology on  
  $\Sp(\R^m) \otimes \Sp(\R^n)$ and $\SpM(\R^m) \otimes \SpM(\R^n)$)
  induced by the inclusion in 
  $\Sp(\R^m \x\R^n)$ and $\SpM(\R^m \x\R^n)$, respectively.    
\end{problem*}

For instance, 
motivated by \cite{Chevet69}, one may
consider the topology on $\Sp(\R^m) \otimes \Sp(\R^n)$ 
  generated by the fundamental system of seminorms
  \begin{align} \label{eq:tensor1}
    h  \mapsto \inf \Big( \|(\|g_i^{(\be)}\|_{L^p(\R^n)})_i\|_{l^p} 
    \sup_{\substack{u \in L^q(\R^m)\\ \|u\|_{L^q(\R^m)}\le 1}} \|(\int_{\R^m} f_i^{(\al)} u \,dx)_i \|_{l^q} \Big),
    \quad \al \in \N^m, \be \in \N^n, 
  \end{align}
  where the infimum is taken over all representations $h=\sum_i f_i \otimes g_i$ ($f_i$ and $g_i$ are zero except for 
  finitely many indices) and $1/p+1/q=1$. By H\"older's inequality,
    \begin{align*}
          \|h^{(\al,\be)}\|_{L^p(\R^m \x \R^n)}^p 
          &= \int_{\R^n} \int_{\R^m} \Big|\sum_{i=1}^k f_i^{(\al)}(x)\,  g_i^{(\be)}(y)\Big|^p \,dx\,dy \\
          &= \int_{\R^n} \sup_{\substack{u \in L^q(\R^m)\\ \|u\|_{L^q(\R^m)}\le 1}} \Big| \int_{\R^m} 
          \sum_{i=1}^k f_i^{(\al)}(x) u(x)\, g_i^{(\be)}(y) \,dx \Big|^p\,dy 
          \\  
          &\le \int_{\R^n}\sum_i |g_i^{(\be)}(y)|^p \,dy  \sup_{\substack{u \in L^q(\R^m)\\ \|u\|_{L^q(\R^m)}\le 1}} 
          \Big( \sum_{i} \Big|\int_{\R^m} 
           f_i^{(\al)}(x) u(x) \,dx \Big|^q \Big)^{\frac{p}{q}} 
          \\
          &= 
          \|(\|g_i^{(\be)}\|_{L^p(\R^n)})_i\|_{l^p}^p 
    \sup_{\substack{u \in L^q(\R^m)\\ \|u\|_{L^q(\R^m)}\le 1}} \|(\int_{\R^m} f_i^{(\al)} u \,dx)_i \|_{l^q}^p. 
    \end{align*}  
    Thus, the topology generated by \eqref{eq:tensor1} is at least as strong 
    as the one induced by the inclusion in $\Sp(\R^m \x \R^n)$. Is it the same?

\section{Failure of exponential law for non-moderate growth} \label{sec:failure}

\subsection{The exponential law fails if $M=(M_k)$ has non-moderate growth}

The $\BrM$, $\SrLM$, $\DrM$, $\SprM$ exponential law (actually the inclusion ($\subseteq$)) fails if 
  $M=(M_k)$ 
  has non-moderate growth:

\label{thm:f}
\begin{theorem*} 
  Let $M=(M_k)$ be a weakly log-convex non-quasianalytic weight sequence with non-moderate growth and 
  let $L=(L_k)$ be a weight sequence satisfying $1 \le k!\, L_k$. 
  Then:
  \begin{itemize}
    \item There exists $f \in \SrLM(\R^2,\C)$ so that $f^\vee :\R \to \SrLM(\R,\C)$ 
      is not $\SrLM$.
    \item There exists $f \in \BrM(\R^2,\C)$ so that $f^\vee :\R \to \BrM(\R,\C)$ 
      is not $\BrM$.
    \item There exists $f \in \DrM(\R^2,\C)$ so that $f^\vee :\R \to \DrM(\R,\C)$ 
      is not $\DrM$.   
  \end{itemize}    
  If furthermore $M=(M_k)$ is derivation closed then:   
  \begin{itemize}
    \item There exists $f \in \SprM(\R^2,\C)$ so that $f^\vee :\R \to \SprM(\R,\C)$ 
      is not $\SprM$.      
  \end{itemize}  
\end{theorem*}

\begin{proof}
  Since $M=(M_k)$ has non-moderate growth, there exist $j_n \nearrow \infty$ and $k_n>0$ 
  such that 
  \begin{equation} \label{eq:ex1}
    \Big(\frac{M_{k_n+j_n}}{M_{k_n} M_{j_n}}\Big)^\frac1{k_n+j_n} \ge n^2\, n!\, L_n.  
  \end{equation}
  Since $M=(M_k)$ is weakly log-convex, there exists $g \in C^{\{M\}}(\R,\C)$
  such that $g^{(k)}(0) = i^k h_k$ and $h_k \ge k! M_k$ for all $k$; see \cite[Thm~1]{Thilliez08}.
  By defining $\tilde f(s,t):= g(s+t)$, we have found $\tilde f\in C^{\{M\}}(\R^2,\C)$ with 
  $\p^\al \tilde f(0,0) = i^{|\al|}  h_{|\al|}$ for all $\al \in \N^2$.  
  Choose a function $\ph \in \cD^{\{M\}}(\R^2,\R)$ that is identically $1$ in a neighborhood of the origin.
  Then $f := \ph \tilde f$ is an element of $\DrM(\R^2,\C)$, $\SrLM(\R^2,\C)$, 
  and of $\BrM(\R^2,\C)$ (by Propositions \ref{prop:incl1} and \ref{prop:incl2}) 
  and satisfies
  \begin{equation} \label{eq:ex2}
    \p^\al f(0,0) = i^{|\al|} h_{|\al|}, \quad h_{|\al|} \ge |\al|!\, M_{|\al|}, \quad \text{ for all } \al \in \N^2.  
  \end{equation}  
  
  \paragraph{\bf Case $\cA = \SrLM$}
  Consider the linear functional $\ell : \SrLM(\R,\C) \to \C$
  given by 
  \[
  \ell(g)=\sum_n\frac{i^{3j_n} g^{(j_n)}(0)}{n! j_n!\,L_n M_{j_n}\,n^{n+j_n}}.
  \]
  This functional is continuous, since 
  $$
  \Big|\sum_n\frac{i^{3j_n} g^{(j_n)}(0)}{n! j_n!\,L_n M_{j_n}\,n^{n+j_n}}\Big| 
  \le 
  \sum_n\frac{|g^{(j_n)}(0)|}{\si^{n+j_n}\, n! j_n!\,L_n M_{j_n}} \Big(\frac{\si}{n}\Big)^{n+j_n}
  \le 
  C(\si)\,\|g\|^{L,M}_{\R,\si} < \infty,
  $$
  for suitable $\si$, where $C(\si):=\sum_n (\frac{\si}{n})^{n+j_n} < \infty$.
  However, $\ell\o f^\vee $ is not $\SrLM$:
  \begin{align*}
  \|\ell\o f^\vee \|^{L,M}_{\R,\si} &= \sup_{\substack{p,q \in \N\\ t\in \R}} \frac{(1+|t|)^p |(\ell \o f^\vee)^{(q)}(t)|}{\si^{p+q}\, p!q!\, L_p M_q}\\
  &\ge \sup_{q \in \N} \frac{1}{\si^{q}\, q!\, M_q} \Big|\sum_n\frac{i^{3j_n}f^{(j_n,q)}(0,0)}{n! j_n!\,L_n M_{j_n}\,n^{n+j_n}}\Big|  
  \hspace{.5cm} (\text{setting } t=p=0) \\
  &= \sup_{q \in \N} \frac{1}{\si^{q}\, q!\, M_q} \Big|\sum_n\frac{i^{4j_n+q} h_{(j_n,q)}}{n! j_n!\,L_n M_{j_n}\,n^{n+j_n}} \Big| 
  \hspace{.5cm} (\text{by \eqref{eq:ex2}}) \\
  &= \sup_{q \in \N} \frac{1}{\si^{q}\, q!\, M_q} \sum_n\frac{h_{(j_n,q)}}{n! j_n!\,L_n M_{j_n}\,n^{n+j_n}}  \\
  &\ge \sup_{n \in \N}  \frac{h_{(j_n,k_n)}}{n!\,L_n\, k_n! j_n!\,  M_{k_n} M_{j_n}\,n^{n+j_n}\, \si^{k_n}}   
  \hspace{.83cm} (\text{setting } q=k_n) \\
  &\ge \sup_{n \in \N}  \frac{(k_n+j_n)!\, M_{k_n+j_n}}{n!\,L_n\, k_n! j_n!\,  M_{k_n} M_{j_n}\,n^{n+j_n}\, \si^{k_n}}   
  \hspace{.83cm} (\text{by \eqref{eq:ex2}}) \\
  &\ge \sup_{n \in \N}  \frac{ (n^2\, n!\,L_n)^{k_n+j_n}}{n!\,L_n\, n^{n+j_n}\, \si^{k_n}}   
  \hspace{3.18cm} (\text{by \eqref{eq:ex1}}) \\
  &\ge \sup_{n \in \N}  \frac{ n^{2k_n}}{\si^{k_n}} 
  \hspace{4.88cm} (\text{as } n!\,L_n\ge 1 \text{ and } j_n\ge n) \\
  &= \infty \hspace{5.93cm} (\text{as } k_n\ge 1),
  \end{align*}
  for all $\si>0$.

  \paragraph{\bf Case $\cA \in \{\BrM,\DrM\}$}
  An analogous computation shows that, for the continuous
  linear functional $\ell : \BrM(\R,\C) \to \C$
  given by 
  \[
  \ell(g)=\sum_n\frac{i^{3j_n} g^{(j_n)}(0)}{j_n!\, M_{j_n}\,n^{j_n}},
  \]
  we have that $\ell\o f^\vee$ is not $\BrM$. The case $\DrM$ follows immediately, since $\ell$ is also a continuous linear functional on $\DrM(\R,\C)$.

  \paragraph{\bf Case $\cA = \SprM$}
  Let $h \in C^\infty(\R^2,\C)$ satisfy $\on{supp} h \subseteq \{|x|< 1\}$. If $a \not\in \on{supp} h$, we have
  \[
    h(a+tx) = \int_{0}^t \p_1 h(a+sx) x_1 + \p_2 h(a+sx) x_2 \, ds
  \]
  and hence, for $t=1$, $a=(-1,0)$ and $x=-a = (1,0)$,
  \begin{align*}
    |h(0)| &\le  \int_{0}^1 |\p_1 h(s-1,0)|\, ds \le  \Big(\int_{0}^1 |\p_1 h(s-1,0)|^p\, ds\Big)^{1/p} \\
    &\le \Big(\int_\R |\p_1 h(x,0)|^p\, dx\Big)^{1/p} \le C \Big(\int_{\R^2} |\p_1 h(x,y)|^p + |\p_1^2 h(x,y)|^p\, d(x,y)\Big)^{1/p},
  \end{align*}
  where the last inequality can be seen as in the proof of Lemma \ref{lem:H1}.
  In the following we apply this to the function $f$ from \eqref{eq:ex2}, where we assume that $\on{supp} \ph \subseteq \{|x|<1\}$. 
  By Proposition \ref{prop:incl2}, $f$ is an element of $\SprM(\R^2,\C)$. 
  For arbitrary $\si,\ta \ge 1$,  
  \begin{align*}
    \sup_{k,j \in \N} &\frac{\int_\R \|\p_2^k[(f^\vee)^{(j)}(x)]\|_{L^p(\R)}^p dx}{(\ta^k\si^j\, j! k!\, M_j M_k)^p}
    =
    \sup_{k,j \in \N} \frac{\int_{\R^2}|\p_2^k \p_1^j f(x,y)|^p\, d(x,y)}{(\ta^k\si^j\, j! k!\, M_j M_k)^p}\\
    &\ge \sup_{n \in \N} \frac{\int_{\R^2}|f^{(j_n+1,k_n)}(x,y)|^p\, d(x,y) + \int_{\R^2}|f^{(j_n+2,k_n)}(x,y)|^p\, d(x,y)}
    {(\ta^{k_n}\si^{j_n+2}\, (j_n+2)! k_n!\, M_{j_n+2} M_{k_n})^p} 
    \\
    &\ge C^{-p} \sup_{n \in \N} \frac{|f^{(j_n,k_n)}(0,0)|^p}
    {(\ta^{k_n}\si^{j_n+2}\, (j_n+2)! k_n!\, M_{j_n+2} M_{k_n})^p} \\
    &\ge \Big(\frac{1}{C\si^2} \sup_{n \in \N} \frac{(k_n+j_n)!\, M_{k_n+j_n}}
    {\ta^{k_n}\si^{j_n}\, (j_n+2)! k_n!\, M_{j_n+2} M_{k_n}}\Big)^p \\
    &\ge \Big(\frac{1}{C_1} \sup_{n \in \N} \frac{(n^2\, n!\,L_n)^{k_n+j_n}}
    {\ta^{k_n}\si_1^{j_n}}\Big)^p \hspace{3cm} \text{by \eqref{eq:ex1} and \thetag{\ref{ssec:ws}.\ref{eq:dc1}}}\\
    &=\infty,
  \end{align*}
  where $C_1$ and $\si_1$ are suitable constants, using that $k!\,M_k$ is non-decreasing (because log-convex) and derivation closed.
  In view of Lemma \ref{lem:WR}, $f^\vee$ is not $\SprM$. 
  \end{proof}

\subsection{The exponential law fails if $L=(L_k)$ has non-moderate growth}

We shall now show that the $\SrLM$-exponential law (the inclusion ($\subseteq$)) also fails if 
$L=(L_k)$ has non-moderate growth; 
see Theorem \ref{thm:f2} below. 
We will use the Fourier transform.

Let $E$ be convenient.
For a function $g \in  \cS(\R,E)$ we define its Fourier transform $\sF g$ 
and its inverse Fourier transform $\bar \sF g$ (see the lemma below) by 
\begin{align*}
  \sF g(\xi) &:= \int_\R g(x) e^{-2\pi i x \xi}\, d x, \quad \bar\sF g(\xi) :=  \sF g(-\xi).
\end{align*}
These integrals exist since integration commutes with continuous linear functionals on $E$.

Let $L=(L_k)$ and $M=(M_k)$ be weakly log-convex, non-decreasing, and derivation closed weight sequences.
We shall use the classical result (see \cite[p.~200]{GelfandShilov68})
\begin{equation} \label{eq:classical} 
  \sF(\SrLM(\R,\C)) = \SrML(\R,\C) \quad\text{ and }\quad \bar \sF(\SrLM(\R,\C)) = \SrML(\R,\C).
\end{equation}
We give a short argument for the inclusion $\sF(\SrLM(\R,\C)) \subseteq \SrML(\R,\C)$ in order to demonstrate that 
the assumptions on $L=(L_k)$ and $M=(M_k)$ are sufficient for \eqref{eq:classical}; in the literature often also  
moderate growth is assumed, but this we want to avoid in view of Theorem \ref{thm:f2}.
By partial integration, 
\begin{align*}
  |\xi^{p} (\sF g)^{(q)}(\xi)| 
  \le (2\pi)^{q-p} \sum_{\ell=0}^{\min \{p,q\}} \frac{p!q!}{\ell! (p-\ell)! (q-\ell)!} 
     \int_\R |x^{q-\ell} g^{(p-\ell)}(x)|\, d x.
\end{align*}
Since $g \in \SrLM(\R,\C)$ and $|x|^{q-\ell} \le (1+|x|)^{q-\ell} \le (1+|x|)^{q}$, there are $C,\si >0$ so that
\begin{align*}
  |\xi^{p} (\sF g)^{(q)}(\xi)| 
  &\le (2\pi)^{q} \int_\R \frac{d x}{(1+|x|)^2} \\
  & \hspace{.5cm} \x \sum_{\ell=0}^{\min \{p,q\}} \frac{p!q!}{\ell! (p-\ell)! (q-\ell)!} C \si^{p+q-\ell+2}\ (q+2)! (p-\ell)!\, L_{q+2} M_{p-\ell}
      \\
  &\le 2C  (2\pi)^{q} (q+2)!   p!\, L_{q+2} M_p \sum_{\ell=0}^{\min \{p,q\}} \frac{q!}{\ell!  (q-\ell)!} \si^{p+q-\ell+2}  
      \\
  &\le 2C  (2\pi)^{q} (q+2)!   p!\, L_{q+2} M_p \, \si^{p+2} (1+\si)^q  
\end{align*} 
as $M=(M_k)$ is non-decreasing. 
Thus, $\sF g \in \SrML(\R,\C)$, since $L=(L_k)$ is derivation closed.

\begin{lemma*}
  We have: 
  \begin{enumerate}
    \item[(2)] If $g \in  \SrLM(\R,E)$ then $\sF g, \bar \sF g \in \SrML(\R,E)$.
    \item[(3)] We have $\bar \sF \o \sF = \sF \o \bar \sF = \Id$ 
    on $\SrLM(\R,E)$.
    \item[(4)] Let $f \in \cS(\R^2,\C)$. Then
  \[
    \sF f(\xi_1,\xi_2) = \sF_2(\sF_1 f^\vee(\xi_1))(\xi_2),
  \] 
  where $\sF_1 : \cS(\R,\cS(\R,\C)) \to \cS(\R,\cS(\R,\C))$ and $\sF_2 : \cS(\R,\C) \to \cS(\R,\C)$ 
  denote the respective Fourier transforms.
  \end{enumerate}
\end{lemma*}

\begin{proof}
  For each $\ell \in E^*$ we have $\ell \o \sF g = \sF(\ell \o g)$ and $\ell \o \bar \sF g = \bar \sF(\ell \o g)$.
  Thus (2) follows from \eqref{eq:classical}. Furthermore, for all $\ell \in E^*$
  \[
    \ell \o g = \bar \sF \sF(\ell \o g) =  \bar \sF (\ell \o \sF g) = \ell \o \bar \sF \sF g
  \]
  which implies (3). Finally,
  \begin{align*}
    \sF f(\xi_1,\xi_2) &= \int_{\R^2} f(x_1,x_2) e^{-2\pi i (x_1 \xi_1 + x_2 \xi_2)} \, d (x_1,x_2) \\
    &= \int_\R \int_\R f(x_1,x_2) e^{-2\pi i x_1 \xi_1} \, d x_1 \,  e^{-2\pi i  x_2 \xi_2} \, d x_2 \\
    &= \int_\R \sF_1 f^\vee(\xi_1)(x_2) \,  e^{-2\pi i  x_2 \xi_2} \, d x_2 \\
    &= \sF_2(\sF_1 f^\vee(\xi_1))(\xi_2),
  \end{align*}
  that is (4).
\end{proof}

\label{thm:f2}
\begin{theorem*} 
  Let $L=(L_k)$ and $M=(M_k)$ be weakly log-convex, non-decreasing, non-quasianalytic, and derivation closed weight sequences. 
  Assume that $L=(L_k)$ has non-moderate growth. 
  Then:
  \begin{itemize}
    \item There exists $g \in \SrLM(\R^2,\C)$ so that $g^\vee :\R \to \SrLM(\R,\C)$ 
      is not $\SrLM$.
  \end{itemize}
\end{theorem*}

\begin{proof}
  Let $f\in \SrML(\R^2,\C)$ be the function from Theorem \ref{thm:f}. 
  Then $f^\vee :\R \to \SrML(\R,\C)$ 
      is not $\SrML$. 
  Set $g:= \sF f \in \SrLM(\R^2,\C)$. 
  Suppose for contradiction that $g^\vee \in \SrLM(\R,\SrLM(\R,\C))$.
  By the above lemma, we have 
  \[
     \sF_2 \o (\sF_1 f^\vee) = g^\vee,
  \]    
  thus 
  \[
     \sF_1 f^\vee = \bar \sF_2 \o g^\vee \in \SrLM(\R,\SrML(\R,\C)),
  \]
  and hence 
  \[
      f^\vee = \bar \sF_1 \o \bar \sF_2 \o g^\vee \in \SrML(\R,\SrML(\R,\C)),
  \]
  a contradiction.
\end{proof}

\section{Stability under composition}  \label{sec:comp}

None of the classes $\cA$ of test functions considered in this paper form categories, since there are no 
identities; no non-zero linear mapping belongs to $\cB$. 
We shall see in this section that $\cB$ and $\BM$ are closed under composition, in contrast to all other cases. 
The following example shows that the ``$0$th derivative'' of the composite $f \o g$ may not have the required decay 
properties at infinity, since $g$ is globally bounded.

\begin{example*}
  Let $f,g \in \cD(\R)$ be such that $f|_{[-1,1]} = 1$ and $|g|\le 1$. 
  The composite $f \o g= 1$ is not in $\bigcup_{1 \le p<\infty} \Sp(\R)$, and hence neither in $\cD(\R)$ and nor in $\cS(\R)$.     
\end{example*}

\subsection{The cases $\cB$ and $\BM$} \label{ssec:Bcomp}

We want to consider mappings of class $\cB$ or $\BM$, but only from the first derivative onwards. 
For Banach spaces $E,F$ and open $U \subseteq E$, we set 
\begin{align*}
  \cB_{\ge 1}(U,F) &:=  \Big\{f \in C^\infty(U,F) : \|f\|^{(k)}_{U} < \infty \text{ for all } k\in \N_{\ge 1}\Big\}, \\
  \BbM_{\ge 1}(U,F) &:= \Big\{f \in C^\infty(U,F) : 
    \forall \rh>0 \sup_{k \in \N_{\ge1}, x\in U} \frac{\|f^{(k)}(x)\|_{\Lin^k(E;F)}}{\rh^k\, k!\,M_k} <\infty \Big\}, \\
  \BrM_{\ge 1}(U,F) &:= \Big\{f \in C^\infty(U,F) : 
    \exists \rh>0 \sup_{k \in \N_{\ge1}, x\in U} \frac{\|f^{(k)}(x)\|_{\Lin^k(E;F)}}{\rh^k\, k!\,M_k} <\infty \Big\}.
\end{align*}
For convenient vector spaces $E,F$ and $c^\infty$-open $U \subseteq E$, let
\begin{align*}
  \cB_{\ge 1}(U,F) &:=  \Big\{f \in C^\infty(U,F) : \forall \ell ~\forall B : \ell \o f \o i_B \in \cB_{\ge1}(U_B,\R)\Big\} \\
  \BM_{\ge 1}(U,F) &:= \Big\{f \in C^\infty(U,F) : \forall \ell ~\forall B : \ell \o f \o i_B \in \BM_{\ge1}(U_B,\R)\Big\},
\end{align*}
where $\ell \in F^*$ and $B \in \sB(E)$. 
Note that $\cB_{\ge 1}$ and $\BM_{\ge 1}$ were denoted $\cB_2$ and $\BM_2$ in \cite{KrieglMichorRainer14a}.  

\begin{definition*}
  Let $\cA \in \{\cB,\BM\}$. 
  An \textbf{$\cA_{\ge1}$-(Banach) plot} in a convenient vector space $F$ is a $\cA_{\ge1}$-mapping 
$g : E \supseteq U \to F$ defined in an open convex subset $U$ of a Banach space $E$.   
\end{definition*}

\begin{proposition*} 
  Let $M=(M_k)$ be a log-convex weight sequence.
  Let $f:U \to F$ be a mapping between a $c^\infty$-open convex subset $U$ of a convenient vector space $E$ and a Banach space $F$.
  Then:
  \begin{align*}
    f \in \cB    &\Longleftrightarrow f \o g \in \cB \text{ for all $\cB_{\ge1}$-plots } g, \\
    f \in \BM    &\Longleftrightarrow f \o g \in \BM \text{ for all $\BM_{\ge1}$-plots } g.
  \end{align*}
\end{proposition*}

\begin{proof}
The direction ($\Leftarrow$) follows from the definition by using $g=i_B$, $B \in \sB(E)$.
For the direction ($\Rightarrow$) let $g: G \supseteq V \to U$ be an $\cA_{\ge1}$-plot, where $\cA \in \{\cB,\BM\}$. 
Fix some point $x_0 \in V$.

\paragraph{\bf Case $\cA= \cB$}
 
  Since $\cB = \cBb$, for all $k \in \N_{\ge1}$ the set 
  $\{g(x_0)\} \cup \bigcup_{1 \le \ell \le k} \set^{(\ell)}_{V}(g)$ 
  is bounded in $E$ and hence contained in some $B_k \in \sB(E)$.
By Fa\`a di Bruno's formula for Banach spaces (see \cite{FaadiBruno1855} for the 1-dimensional version), 
we find for $k\ge 1$,   
\begin{align} \label{eq:Faa}
\begin{split}
  &\frac{\|(f\o g)^{(k)}(x)\|_{\Lin^k(G;F)}}{k!} \le \\
&\le \sum_{j\ge 1} \sum_{\substack{\al\in \N_{>0}^j\\ \al_1+\dots+\al_j =k}}
\frac{\|f^{(j)}(g(x))\|_{\Lin^j(E_{B_k};F)}}{j!}\;\prod_{i=1}^j\;
\frac{\|g^{(\al_i)}(x)\|_{\Lin^{\al_i}(G;E_{B_k})}}{\al_i!}
\end{split}
\end{align}
Since $g' : V \to \Lin(G,E_{B_k})$ we have $g(V) \subseteq E_{B_k}$ (by integration as $g(x_0) \in B_k$ and $V$ is convex), and thus
taking the supremum over $x \in V$, we deduce 
\begin{align*}
  \|f\o g\|^{(k)}_{V} \le k! \sum_j \sum_\al \frac{\|f\|^{(j)}_{U_{B_k}}}{j!} \prod_i \frac{\|g\|^{(\al_i)}_{V}}{\al_i!} < \infty 
\end{align*}
for each $k \ge 1$. For $k=0$ we have 
\[
  \|f\o g\|^{(0)}_{V} \le \|f\|^{(0)}_{U_{B_k}} < \infty.
\]

\paragraph{\bf Case $\cA = \cB^{\bM}$} 
  
  Since $\cB^{\bM} = \cBb^{\bM}$, for all $\rh>0$ the set 
  \begin{align} \label{eq:setM1}
    \set^M_{V,\rh,\ge1}(g) := \Big\{\frac{g^{(k)}(x)(v_1,\dots,v_k)}{k!\,\rh^k\, M_k}:k\in \N_{\ge 1},x\in V,\|v_i\|_G\leq 1\Big\}
  \end{align}
  is bounded in $E$ and hence $\{g(x_0)\} \cup \set^M_{V,\rh,\ge1}(g)$ is contained in some $B \in \sB(E)$.
  (Note that $\set^M_{V,\rh}(g) = \set^M_{V,\rh,\ge1}(g) \cup \{g(x) : x \in V\}$.)
By \eqref{eq:Faa} (with $B_k$ replaced by $B$), \thetag{\ref{ssec:ws}.\ref{eq:FdB}}, and since again $g(V) \subseteq E_B$, 
we find    
\begin{align} 
   &\frac{\|(f\o g)^{(k)}(x)\|_{\Lin^k(G;F)}}{k!M_k} \le \nonumber \\
&\le \sum_{j\ge 1} M_1^j \!\!\!\!\sum_{\substack{\al\in \N_{>0}^j\\ \al_1+\dots+\al_j =k}}
\frac{\|f^{(j)}(g(x))\|_{\Lin^j(E_B;F)}}{j!M_j}\;\prod_{i=1}^j\;
\frac{\|g^{(\al_i)}(x)\|_{\Lin^{\al_i}(G;E_B)}}{\al_i!M_{\al_i}}
\nonumber \\ \label{eq:Faacompute}
&\le  M_1 C_f C_g \rh_f \rh_g^k  \sum_{j\ge 1} \binom{k-1}{j-1} (M_1  \rh_f C_g)^{j-1}  
= M_1 C_f C_g \rh_f \rh_g^k (1 + M_1  \rh_f C_g)^{k-1}.
\end{align}
Given $\rh>0$ take $\si>0$ so that $\rh = \sqrt \si + \si$ and set $\rh_g = \sqrt \si$ and 
$\rh_f = (C_g M_1)^{-1} \sqrt \si$. Then $\|f \o g\|^M_{V,\rh}< \infty$. 

\paragraph{\bf Case $\cA=\cB^{\rM}$}
Fix a sequence $(r_k) \in \sR'$. By Proposition \ref{sec:proj} the set  
\[
  \set^M_{V,(r_k\, 2^k),\ge1}(\ell \o g) 
  := \Big\{r_k\, 2^k \frac{(\ell \o g)^{(k)}(x)(v_1,\dots,v_k)}{k!\, M_k} : k\in \mathbb N_{\ge1},x\in V,\|v_i\|_G\leq 1\Big\}
\]
is bounded 
in $\R$ for each $\ell \in E^*$. Thus the set $\{g(x_0)\} \cup \set^M_{V,(r_k\, 2^k),\ge1}(g)$ is contained in some $B \in \sB(E)$ and 
so, for $k\ge 1$, 
\[
\frac{\|g^{(k)}(a)\|_{\Lin^k(G;E_B)} \, r_k}{k!M_k}
\leq \frac1{2^k}. 
\]
Fa\`a di Bruno's formula \eqref{eq:Faa}, \thetag{\ref{ssec:ws}.\ref{eq:FdB}}, and $g(V) \subseteq E_B$ (as before) then give
\begin{equation} \label{eq:Faacompute2}
  \frac{\|(f\o g)^{(k)}(x)\|_{\Lin^k(G;F)}}{k!M_k}\, r_k  
\le \frac{C_f}{2^k} \sum_{j\ge 1} \binom{k-1}{j-1} (M_1\rh_f)^j 
\le \frac{M_1 C_f \rh_f}{1+M_1\rh_f}  \Big(\frac{1+M_1\rh_f}{2}\Big)^k.
\end{equation}
Thus $\set^M_{V,(r_k \de^k)}(f\o g)$ for $\de = 2 (1+M_1\rh_f)^{-1}$ is bounded in $F$. Proposition \ref{sec:proj}  
implies the statement.
\end{proof}

\begin{remark*}
  In particular, for a convenient vector space $E$, Banach spaces $F,G$, and $c^\infty$-open subsets $U \subseteq E$, $V \subseteq G$
  we have 
  \begin{align*}
    f \in \cB(U,F),~ g \in \cB(V,U)    &\Longrightarrow f \o g \in \cB(V,F), \\
    f \in \BM(U,F),~ g \in \BM(V,U)    &\Longrightarrow f \o g \in \BM(V,F).
  \end{align*}
  Note that here we need not assume convexity of $U$ and $V$ because $g(V)$ is bounded by assumption. 
\end{remark*}

\begin{theorem*} 
  Let $M=(M_k)$ be a log-convex weight sequence.
  Let $E,F,G$ be convenient vector spaces, let $U \subseteq E$ and $V \subseteq F$ be $c^\infty$-open.
  Then:
  \begin{align*}
    f \in  \cB(V,G),~ g \in \cB(U,V) &\Longrightarrow  f \o g \in  \cB(V,G), \\
    f \in  \BM(V,G),~ g \in \BM(U,V) &\Longrightarrow  f \o g \in  \BM(V,G). 
  \end{align*}
\end{theorem*}

\begin{proof}
  We must show that for all $B \in \sB(E)$ 
and for all $\ell \in G^*$ the composite
$\ell \o f \o g \o i_B : U_B \to \R$ belongs to $\cA$.
  \[
  \xymatrix{
  U \ar[rr]^{g} && V \ar[rr]^f \ar[drr]_{\ell \o f} && G \ar[d]^{\ell} \\
  U_B \ar[u]^{i_B} \ar[urr]_{g \o i_B} &&&& \R
  }
  \]
By assumption, $g \o i_B$ and $\ell \o f$ are $\cA$.
So the assertion follows from the remark.    
\end{proof}

The proofs of the above proposition and theorem imply the following corollary.

\begin{corollary*} 
  Let $M=(M_k)$ be a log-convex weight sequence.
  Let $E,F,G$ be convenient, let $U \subseteq E$ and $V \subseteq F$ be $c^\infty$-open, and let $V$ be convex.
  Then:
  \begin{align*}
    f \in  \cB_{\ge1}(V,G),~ g \in \cB_{\ge1}(U,V) &\Longrightarrow  f \o g \in \cB_{\ge1}(U,G), \\
    f \in  \BM_{\ge1}(V,G),~ g \in \BM_{\ge1}(U,V) &\Longrightarrow  f \o g \in \BM_{\ge1}(U,G). 
  \end{align*}
\end{corollary*}

Thus, the $\cB_{\ge1}$-mappings between convenient vector spaces form a category, and, if $M=(M_k)$ is log-convex, 
then the $\BM_{\ge1}$-mappings between convenient vector spaces form a category. 
However, these categories are not cartesian closed as seen by the following example.

\begin{example*}
  The function $f : \R^2 \to \R, (x,y) \mapsto xy$ is not $\cB_{\ge1}$, since 
  $f'(x,y) = (\begin{matrix}
      y & x
    \end{matrix})$ 
  is not globally bounded on $\R^2$. 
  However, $f^\vee : x \mapsto (y \mapsto xy)$ has values in $\cB_{\ge1}(\R,\R)$ and
  is $\cB_{\ge1}$. In fact, $(f^\vee)'$ is the constant $\Id \in \cB_{\ge1}(\R,\R)$ and 
  higher derivatives vanish.
\end{example*}

\subsection{The cases $\cS$ and $\SLM$} \label{ssec:Scomp}

For Banach spaces $E,F$ we set 
\begin{align*}
  \cS_{\ge1}(E,F) &:=  \Big\{f \in C^\infty(E,F) : \|f\|^{(k,\ell)}_{E} < \infty \text{ for all } k\in \N, \ell\in \N_{\ge 1}\Big\}, \\
  \SbLMone(E,F) &:= \Big\{f \in C^\infty(E,F) : \\
  & \hspace{1.8cm} 
    \forall \si>0 \sup_{k\in \N, \ell \in \N_{\ge1}, x\in E} \frac{(1+\|x\|)^k \|f^{(\ell)}(x)\|_{\Lin^\ell(E;F)}}{\si^{k+\ell}\, k!\ell!\,L_k M_\ell} <\infty \Big\}, \\
  \SrLMone(E,F) &:= \Big\{f \in C^\infty(E,F) : \\
  & \hspace{1.8cm} 
    \exists \si>0 \sup_{k\in \N, \ell \in \N_{\ge1}, x\in E} \frac{(1+\|x\|)^k \|f^{(\ell)}(x)\|_{\Lin^\ell(E;F)}}{\si^{k+\ell}\, k!\ell!\,L_k M_\ell} <\infty \Big\}.
\end{align*}
For convenient vector spaces $E,F$, let
\begin{align*}
  \cS_{\ge1}(E,F) &:=  \Big\{f \in C^\infty(E,F) : \forall \ell ~\forall B : \ell \o f \o i_B \in \cS_{\ge1}(E_B,\R)\Big\} \\
  \SLMone(E,F) &:= \Big\{f \in C^\infty(E,F) : \forall \ell ~\forall B : \ell \o f \o i_B \in \SLMone(E_B,\R)\Big\},
\end{align*}
where $\ell \in F^*$ and $B \in \sB(E)$.

\begin{theorem*} 
  Let $M=(M_k)$ and $L=(L_k)$ be weight sequences and assume that $M$ is log-convex. 
  Let $E,F,G$ be convenient.
  We have:
  \begin{align*}
    f \in  \cB_{\ge1}(F,G),~ g \in \cS_{\ge1}(E,F) &\Longrightarrow  f \o g \in \cS_{\ge1}(E,G), \\
    f \in  \BM_{\ge1}(F,G),~ g \in \SLMone(E,F) &\Longrightarrow  f \o g \in \SLMone(E,G). 
  \end{align*}
\end{theorem*}

\begin{corollary*}
  The $\cS_{\ge1}$-mappings between convenient vector spaces form a category. If $M=(M_k)$ is log-convex, 
  then also the $\SLMone$-mappings between convenient vector spaces form a category. 
  Neither of this categories in cartesian closed, by Example \ref{ssec:Bcomp}.
\end{corollary*}

\begin{proof}
Let $\cA \in \{\cS,\SLM\}$.
We must show that for all $B \in \sB(E)$ 
and for all $\ell \in G^*$ the composite
$\ell \o f \o g \o i_B : E_B \to \R$ belongs to $\cA_{\ge1}$.
  \[
  \xymatrix{
  E \ar[rr]^{g} && F \ar[rr]^f \ar[drr]_{\ell \o f} && G \ar[d]^{\ell} \\
  E_B \ar[u]^{i_B} \ar[urr]_{g \o i_B} &&&& \R
  }
  \]
Thus it suffices to show the assertions under the assumption that $E$ and $G$ are Banach spaces 
which we adopt for the rest of the proof.

\paragraph{\bf Case $\cA= \cS$}

  Since $\cS = \cSb$, for all $p,q \in \N$ the set 
  $\{g(0)\} \cup \bigcup_{1 \le \ell \le q} \set^{(0,\ell)}_{E}(g)$  
  is bounded in $F$ and hence contained in some $B_{q} \in \sB(F)$.
By Fa\`a di Bruno's formula, 
we find for $q\ge 1$,   
\begin{align} \label{eq:FaaS} 
\begin{split}
  &\frac{\|(f\o g)^{(q)}(x)\|_{\Lin^q(E;G)}}{q!} \le \\
&\le \sum_{j\ge 1} \sum_{\substack{\al\in \N_{>0}^j\\ \al_1+\dots+\al_j =q}}
\frac{\|f^{(j)}(g(x))\|_{\Lin^j(F_{B_{q}};G)}}{j!}\;\prod_{i=1}^j\;
\frac{\|g^{(\al_i)}(x)\|_{\Lin^{\al_i}(E;F_{B_{q}})}}{\al_i!}.
\end{split}
\end{align}
Since $g' : E \to \Lin(E;F_{B_{q}})$ we have $g(E) \subseteq F_{B_{q}}$ (by integration as $g(0) \in B_{q}$). 
Multiplying both sides with $(1+\|x\|_E)^p$ and  
taking the supremum over $x \in E$, we deduce 
\begin{align*}
  \|f\o g\|^{(p,q)}_{E} \le q! \sum_j \sum_\al \frac{\|f\|^{(0,j)}_{F_{B_{q}}}}{j!} 
  \prod_{i=1}^{j-1} \frac{\|g\|^{(0,\al_i)}_{E}}{\al_i!} \frac{\|g\|^{(p,\al_j)}_{E}}{\al_j!} < \infty 
\end{align*}
for each $p \in \N$, $q \in \N_{\ge1}$. 

\paragraph{\bf Case $\cA= \SbLM$}

  Since $\SbLM = \SbLMb$, for all $\rh>0$ the set $\set^{M}_{E,\rh,\ge1}(g)$ (defined in (\ref{ssec:Bcomp}.\ref{eq:setM1}))
  is bounded in $F$ and hence $\{g(0)\} \cup \set^{M}_{E,\rh,\ge1}(g)$ is contained in some $B \in \sB(F)$.
By \eqref{eq:FaaS} (with $B_{q}$ replaced by $B$), \thetag{\ref{ssec:ws}.\ref{eq:FdB}}, and since again $g(E) \subseteq F_B$, 
we find    
\begin{align*}
    &\frac{(1+\|x\|_E)^p \|(f\o g)^{(q)}(x)\|_{\Lin^q(E;G)}}{p! q!\, L_p M_q} \le \\
&\le \sum_{j\ge 1} M_1^j \sum_{\substack{\al\in \N_{>0}^j\\ \al_1+\dots+\al_j =q}}
\frac{\|f^{(j)}(g(x))\|_{\Lin^j(F_{B};G)}}{j!\, M_j}\;\prod_{i=1}^{j-1}\;
\frac{\|g^{(\al_i)}(x)\|_{\Lin^{\al_i}(E;F_{B})}}{\al_i!\, M_{\al_i}} \\
& \hspace{6cm} \times 
\frac{(1+\|x\|_E)^p \|g^{(\al_j)}(x)\|_{\Lin^{\al_j}(E;F_{B})}}{p! \al_j!\, L_p M_{\al_j}} 
\\
&\le M_1 C_f C_g \rh_f \rh_g^{p+q} (1+M_1 \rh_f C_g )^{q-1}, 
\end{align*}
by the computation in (\ref{ssec:Bcomp}.\ref{eq:Faacompute}).
Given $\rh>0$ choose $\si>0$ so that $\rh = \sqrt \si + \si$ and set $\rh_g = \sqrt \si$ and 
$\rh_f = (C_g  M_1)^{-1} \sqrt \si$. Then 
\[
 \sup_{x \in E, p\in \N, q \in \N_{\ge 1}} \frac{(1+\|x\|_E)^p \|(f\o g)^{(q)}(x)\|_{\Lin^q(E;G)}}{\rh^{p+q}\, p! q!\, L_p M_q} < \infty.
\]

\paragraph{\bf Case $\cA= \SrLM$}

Fix a sequence $(r_k) \in \sR'$. By Proposition \ref{sec:proj} the set  
\begin{multline*}
  \set^{L,M}_{E,(r_k\, 2^k),\ge1}(\ell \o g) := \Big\{r_{p+q}\, 2^{p+q} \frac{(1+\|x\|_E)^p\, (\ell \o g)^{(q)}(x)(v_1,\dots,v_q)}{p! q!\, L_p M_q}  : 
  \\
  p \in \N, q\in \N_{\ge1},x\in E,\|v_i\|_E\leq 1\Big\}
\end{multline*}
is bounded 
in $\R$ for each $\ell \in F^*$. Thus the set $\{g(0)\} \cup \set^{L,M}_{E,(r_k\, 2^k),\ge 1}(g)$ is contained in some $B \in \sB(F)$ and 
so, for $p \in \N$, $q\in \N_{\ge1}$, and $x \in E$, 
\[
\frac{(1+\|x\|_E)^p\,\|g^{(q)}(x)\|_{\Lin^q(E;F_B)} \, r_{p+q}}{p!q!\, L_p M_q}
\leq \frac1{2^{p+q}}. 
\]
Fa\`a di Bruno's formula \eqref{eq:FaaS}, \thetag{\ref{ssec:ws}.\ref{eq:FdB}}, and $g(E) \subseteq F_B$ (as before) then give
\begin{align*}
    &\frac{(1+\|x\|_E)^p \|(f\o g)^{(q)}(x)\|_{\Lin^q(E;G)}}{p! q!\, L_p M_q} r_{p+q} \le \\
&\le \sum_{j\ge 1} M_1^j \sum_{\substack{\al\in \N_{>0}^j\\ \al_1+\dots+\al_j =q}}
\frac{\|f^{(j)}(g(x))\|_{\Lin^j(F_{B};G)}}{j!\, M_j}\;\prod_{i=1}^{j-1}\;
\frac{\|g^{(\al_i)}(x)\|_{\Lin^{\al_i}(E;F_{B})}\, r_{\al_i}}{\al_i!\, M_{\al_i}} \\
& \hspace{6cm} \times 
\frac{(1+\|x\|_E)^p \|g^{(\al_j)}(x)\|_{\Lin^{\al_j}(E;F_{B})}\, r_{p+\al_i}}{p! \al_j!\, L_p M_{\al_j}} 
\\
&\le \frac{M_1 C_f \rh_f}{1+ M_1 \rh_f} \Big(\frac{1+ M_1 \rh_f}{2}\Big)^{p+q}, 
\end{align*}
by the computation in (\ref{ssec:Bcomp}.\ref{eq:Faacompute2}).
Thus $\set^{L,M}_{E,(r_k \de^k),1}(f\o g)$ for $\de = 2 (1+M_1\rh_f)^{-1}$ is bounded in $G$. Proposition \ref{sec:proj}  
implies that $f\o g \in \SrLMone(E,G)$; 
it suffices to take $a_{p,q}$  as defined in (\ref{sec:proj}.\ref{eq:projS}) for $q \ge1$, set $a_{p,0}:=0$, and apply Lemma \ref{sec:proj}.
\end{proof}


In Section \ref{sec:diff} we also need the following result.

\begin{proposition*}[{\cite[Thm 6.3]{KrieglMichorRainer14a}}] 
  Let $M=(M_k)$ and $L=(L_k)$ be weight sequences and assume that $M$ is log-convex.
  If $g : \R^n \to \R^n$ is a $C^\infty$-diffeomorphism satisfying $g(x) = x + o(x)$ as $|x| \to \infty$, we have the following implications: 
  \begin{align*}
    f \in  \cS(\R^n,\R^n),~ g \in \cB_{\ge1}(\R^n,\R^n) &\Longrightarrow  f \o g \in \cS(\R^n,\R^n), \\
    f \in  \SLM(\R^n,\R^n),~ g \in \BM_{\ge1}(\R^n,\R^n) &\Longrightarrow  f \o g \in \SLM(\R^n,\R^n). 
  \end{align*}
\end{proposition*}

\subsection{The cases $\Sp$ and $\SpM$} \label{ssec:Wcomp}

We set 
\begin{align*}
  \Sp_{\ge1}(\R^m) &:= \Big\{f \in C^\infty(\R^m) : \|f^{(\al)}\|_{L^p(\R^m)} < \infty \text{ for all } \al \in \N^m, |\al|\ge 1\Big\}, \\
  \SpbM_{\ge1}(\R^m) &:=  \Big\{f \in C^\infty(\R^m) :  
    ~\forall \si>0 \sup_{\al \in \N^m, |\al|\ge 1} \frac{\|f^{(\al)}\|_{L^p(\R^m)}}{\si^{|\al|}\,|\al|!\, M_{|\al|}} <\infty \Big\}, \\ 
  \SprM_{\ge1}(\R^m) &:=  \Big\{f \in C^\infty(\R^m) :  
    ~\exists \si>0 \sup_{\al \in \N^m, |\al|\ge 1} \frac{\|f^{(\al)}\|_{L^p(\R^m)}}{\si^{|\al|}\,|\al|!\, M_{|\al|}} <\infty \Big\},
\end{align*}
and 
\[
  \Sp_{\ge1}(\R^m,\R^n) := (\Sp_{\ge1}(\R^m))^n, \qquad \SpM_{\ge1}(\R^m,\R^n) := (\SpM_{\ge1}(\R^m))^n. 
\]

\begin{theorem*} 
  Let $M=(M_k)$ be a log-convex derivation closed weight sequence. Then:
  \begin{align*}
    f \in  \cB_{\ge1}(\R^m,\R^n),~ g \in \Sp_{\ge1}(\R^\ell,\R^m) &\Longrightarrow  f \o g \in \Sp_{\ge1}(\R^\ell,\R^n), \\
    f \in  \BM_{\ge1}(\R^m,\R^n),~ g \in \SpM_{\ge1}(\R^\ell,\R^m) &\Longrightarrow  f \o g \in \SpM_{\ge1}(\R^\ell,\R^n). 
  \end{align*}
\end{theorem*}

\begin{proof} 
	For simplicity we assume that $\ell=m =n=1$; the general case will follow by the same arguments from the 
	Fa\`a di Bruno's formula for partial derivatives.  			
	For $h\ge 1$,
  \begin{align*}
    \frac{\|(f \o g)^{(h)}\|_{L^p(\R)}}{h!} &\le \sum_{j \ge 1} \sum_{\substack{\al \in \N_{>0}^j\\\sum_i \al_i=h}}
    \frac{\|f^{(j)}\|_{L^\infty(\R)}}{j!} \frac{\|g^{(\al_1)}\|_{L^p(\R)}}{\al_1!} \prod_{i=2}^{j} \frac{\|g^{(\al_i)}\|_{L^\infty(\R)}}{\al_i!}  
  \end{align*}
	since $\Sp_{\ge1}(\R) \subseteq \cB_{\ge1}(\R)$. This shows the first part of the theorem.

  For the second part we may argue as follows. 
  By the general Sobolev inequalities there exist $k \in \N_{\ge1}$ a constant $C>0$, both depending only on $p$, so that  
  \[
    \|g^{(j)}\|_{L^\infty(\R)} \le C \|g^{(j)}\|_{W^{k,p}(\R)}.
  \]   
  Using that $M=(M_k)$ is derivation closed and thus (\ref{ssec:ws}.\ref{eq:dc1}), we further have 
  \begin{align*}
    \|g^{(j)}\|_{W^{k,p}(\R)} &= \sum_{i=0}^k \|g^{(j+i)}\|_{L^p(\R)}
    \le C_g \sum_{i=0}^k \rh_g^{j+i} (j+i)! M_{j+i}  
    \le \tilde C_g \tilde \rh_g^j  j! M_{j}. 
  \end{align*}
  This permits to conclude the proof in the same way as the one of Proposition~\ref{ssec:Bcomp}.
\end{proof}

In Section \ref{sec:diff} we also need the following result.

\begin{proposition*}[{\cite[Thm 6.2]{KrieglMichorRainer14a}}] 
  Let $M=(M_k)$ be a log-convex weight sequence. 
  If $g : \R^n \to \R^n$ is a $C^\infty$-diffeomorphism satisfying $\inf_{x\in \R^n} |\det dg(x)|>0$, we have the following implications: 
  \begin{align*}
    f \in  \Sp(\R^n,\R^n),~ g \in \cB_{\ge1}(\R^n,\R^n) &\Longrightarrow  f \o g \in \Sp(\R^n,\R^n), \\
    f \in  \SpM(\R^n,\R^n),~ g \in \BM_{\ge1}(\R^n,\R^n) &\Longrightarrow  f \o g \in \SpM(\R^n,\R^n). 
  \end{align*}
\end{proposition*}

\subsection{The cases $\cD$ and $\DM$} \label{ssec:Dcomp}

We define 
\begin{align*}
  \cD_{\ge 1}(\R^m,\R^n) &:= \{f \in C^\infty(\R^m,\R^n) : f^{(\al)} \in \cD(\R^m,\R^n) \text{ for all } \al \in \N^m, |\al|\ge 1\}, \\
  \DM_{\ge 1}(\R^m,\R^n) &:= \cD_{\ge 1}(\R^m,\R^n) \cap \BM_{\ge 1}(\R^m,\R^n).
\end{align*}

\begin{theorem*} 
  Let $M=(M_k)$ be a log-convex weight sequence. Then:
  \begin{align*}
    f \in  \cB_{\ge1}(\R^m,\R^n),~ g \in \cD_{\ge1}(\R^\ell,\R^m) &\Longrightarrow  f \o g \in \cD_{\ge1}(\R^\ell,\R^n), \\
    f \in  \BM_{\ge1}(\R^m,\R^n),~ g \in \DM_{\ge1}(\R^\ell,\R^m) &\Longrightarrow  f \o g \in \DM_{\ge1}(\R^\ell,\R^n). 
  \end{align*}
\end{theorem*}

\begin{proof}
  By Corollary \ref{ssec:Bcomp}, only the condition on the support must be checked. It follows 
  easily from the chain rule.
\end{proof}

Another immediate consequence of Corollary \ref{ssec:Bcomp} is the following.

\begin{proposition*}
  Let $M=(M_k)$ be a log-convex weight sequence. 
  If $g : \R^n \to \R^n$ is a $C^\infty$-diffeomorphism, 
  we have the following implications: 
  \begin{align*}
    f \in  \cD(\R^n,\R^n),~ g \in \cB_{\ge1}(\R^n,\R^n) &\Longrightarrow  f \o g \in \cD(\R^n,\R^n), \\
    f \in  \DM(\R^n,\R^n),~ g \in \BM_{\ge1}(\R^n,\R^n) &\Longrightarrow  f \o g \in \DM(\R^n,\R^n). 
  \end{align*}
\end{proposition*}

\section{Application: Groups of diffeomorphisms on \texorpdfstring{$\R^n$}{Rn}} \label{sec:diff}

Let $\cA \in \{\cD, \cS, \Sp,\cB, \DM, \SLM,\SpM,\BM\}$ and set
\[
  \Diff\cA = \DiffA := \big\{F=\Id+f: f\in \cA(\R^n,\R^n), \inf_{x \in \R^n}\det(\mathbb I_n + df(x)) >0\big\}.
\]
It was shown in \cite{MichorMumford13} that 
the groups of diffeomorphisms ($1 \le p<q<\infty$)
\[
      \xymatrix{
        \Diff\cD \ar@{{ >}->}[r] & \Diff\cS \ar@{{ >}->}[r] & \Diff W^{\infty,p} \ar@{{ >}->}[r] & \Diff W^{\infty,q} \ar@{{ >}->}[r] & \Diff\cB 
      }
\]
are $C^\infty$-regular Lie groups. The arrows describe $C^\infty$ injective group homomorphisms.
Each group is a normal subgroup of the groups on its right.
In \cite{KrieglMichorRainer14a} we proved that, 
provided that the weight sequence $M=(M_k)$ is log-convex, has moderate growth, and in the Beurling $\CbM \supseteq C^\om$, 
and that $L=(L_k)$ satisfies $L_k \ge 1$ for all $k$, 
the groups of $\CM$-diffeomorphisms
\[
      \xymatrix{
        \Diff\DM \ar@{{ >}->}[r] & \Diff\SLM \ar@{{ >}->}[r] & \Diff\SpM \ar@{{ >}->}[r] & \Diff W^{[M],q} \ar@{{ >}->}[r] & \Diff\BM
      }
\]
  are $\CM$-regular Lie groups. The arrows describe $\CM$ injective group homomorphisms.
  Each group is a normal subgroup in the groups on its right.      

This was done by 
\begin{itemize}
	\item characterizing the $C^\infty$-curves in the space $\cA(\R^n,\R^n)$ for $\cA\in \{\cD, \cS, \Sp,\cB\}$, 
	and the $\CM$-plots in the space $\cA(\R^n,\R^n)$ for $\cA \in \{\DM, \SLM,\SpM,\BM\}$, respectively, and
	\item proving via this characterization that $C^\infty$-curves and $\CM$-plots, respectively, are preserved 
	by the group operations, that is composition and inversion. 
\end{itemize}
The first step is based on the $C^\infty$ and $\CM$ exponential law while the 
second step required a careful application of Fa\`a di Bruno's formula.

In this section we apply the exponential laws established in this paper to conclude in a simpler way that 
$\DiffD$, 
$\DiffS$, $\DiffSp$, $\DiffB$ are $C^\infty$ Lie groups as well as 
that 
$\Diff\DrM(\R^n)$,
$\Diff \SrLM(\R^n)$, $\Diff \SprM(\R^n)$, $\Diff \BrM(\R^n)$ are $\CrM$ Lie groups provided that 
$M=(M_k)$ is non-quasianalytic.

%

\subsection{$\DiffD$, $\DiffS$, $\DiffSp$, $\DiffB$ are $C^\infty$ Lie groups} \label{ssec:smooth}

Let us recall a well-known lemma; for a proof see, e.g., \cite[8.4]{KrieglMichorRainer14a}.

\begin{lemma*} \label{lem:LA}
  If $A \in \on{GL(n)}$ then $\|A^{-1}\| \le |\det A|^{-1} \|A\|^{n-1}$.
\end{lemma*}

\begin{proposition*} \label{prop:inv}
  Let $F \in \cB_{\ge 1}(\R^n,\R^n)$ be a diffeomorphism of $\R^n$ satisfying $\inf_{x \in \R^n}\det dF(x) >0$.
  Then $F^{-1} \in \cB_{\ge 1}(\R^n,\R^n)$ and $\inf_{x \in \R^n}\det dF^{-1}(x) >0$.
\end{proposition*}

\begin{proof}
  Set $G := F^{-1}$. Since $F \o G = \Id$ we have 
  \[
    \det dG(x) = (\det dF(G(x)))^{-1},
  \]
  and so $\|G\|^{(1)}_{\R^n} < \infty$ and $\inf_{x \in \R^n}\det dG(x) >0$, in view of Lemma \ref{lem:LA}.
  Fix $a \in \R^n$ and set $b=F(a)$ and $T=F'(a)^{-1} = G'(b)$. Defining 
  \begin{equation} \label{eq:ph}
    \ph := \Id - T \o F,  
  \end{equation}
  we have 
  \begin{equation} \label{eq:G}
    G = T + \ph \o G.
  \end{equation}
  By Fa\`a di Bruno's formula, for $k\ge2$
  \begin{align*}
    \frac{G^{(k)}(x)}{k!} &= \frac{\ph^{(1)}(G(x)) \o G^{(k)}(x)}{k!} \\
    & \quad+ \on{sym} \sum_{j\ge 2} \sum_{\substack{\al_i>0\\ \al_1+\cdots+\al_j=k}} 
    \frac{\ph^{(j)}(G(x))}{j!} \o \Big(\frac{G^{(\al_1)}(x)}{\al_1!} \times \cdots \times \frac{G^{(\al_j)}(x)}{\al_j!} \Big) \\
    &= \frac{G^{(k)}(x) - T \o F'(G(x))\o G^{(k)}(x)}{k!} \\
    & \quad+ \on{sym} \sum_{j\ge 2} \sum_{\substack{\al_i>0\\ \al_1+\cdots+\al_j=k}} 
    \frac{\ph^{(j)}(G(x))}{j!} \o \Big(\frac{G^{(\al_1)}(x)}{\al_1!} \times \cdots \times \frac{G^{(\al_j)}(x)}{\al_j!} \Big) 
  \end{align*}
  and hence we can conclude by induction that $\|G\|^{(k)}_{\R^n}< \infty$ for all $k\ge 1$, since
  $(F')^{-1}$ is globally bounded by assumption. 
\end{proof}

Let $\cA \in \{\cD, \cS, \Sp, \cB\}$. 
The elements of $\DiffA$ are smooth diffeomorphisms on $\R^n$, since they are surjective proper submersions (cf.~\cite{MichorMumford13}).
The composite of two elements in $\DiffA$ is in $\DiffA$, by Propositions \ref{ssec:Bcomp}, \ref{ssec:Scomp}, \ref{ssec:Wcomp}, \ref{ssec:Dcomp}, and
\ref{prop:incl2}.

If $F = \Id+f\in \DiffB$, then $G = \Id + g := F^{-1} \in \cB_{\ge1}(\R^n,\R^n)$ and $\inf_{x \in \R^n}\det dG(x) >0$, 
by Proposition \ref{prop:inv}. This implies that $g \in \cB_{\ge1}(\R^n,\R^n)$ and clearly also $g$ itself is globally bounded, 
as follows for instance from
\begin{equation} \label{eq:group2}
    (\Id+g)\o(\Id+f)=\Id  \quad\Longleftrightarrow\quad f(x)+g(x+f(x))=0.
\end{equation}
So $F^{-1} \in \DiffB$, and $\DiffB$ forms a group.

If $F = \Id+f\in \DiffA$ for $\cA\in \{\cD, \cS, \Sp\}$ then $F \in \DiffB$, and hence $g = F^{-1} - \Id \in \cB(\R^n,\R^n)$, by the 
previous paragraph. We then conclude that $g \in \cA(\R^n,\R^n)$, by applying Propositions \ref{ssec:Scomp}, \ref{ssec:Wcomp}, and \ref{ssec:Dcomp} 
to
\begin{equation}
  (\Id+f)\o(\Id+g)= \Id \quad\Longleftrightarrow\quad  g(x)+f(x+g(x))=0.
\end{equation}
So $\DiffA$ forms a group.

We shall now show that the group operations are $C^\infty$, or equivalently, that they preserve $C^\infty$-curves. 
Actually, it suffices that they take $C^\infty$-curves with compact support to $C^\infty$-curves; in the special curve lemma 
\cite[2.8]{KM97} the curve may be chosen with compact support. We will take 
advantage of this fact here.
By the $\cA$ exponential law, Theorems \ref{thm:Bexp}, \ref{thm:Dexp}, and \ref{thm:Wexp}, and by Proposition \ref{prop:incl2}, we have the diagram 
\begin{align} \label{diag:A} 
\begin{split}
  \xymatrix{
    \cD(\R,\cA(\R^n,\R^n)) \ar@{{ >}->}[r] \ar[dr] & \cA(\R,\cA(\R^n,\R^n)) \ar@{{ >}->}[r] & C^\infty(\R,\cA(\R^n,\R^n)) \\
    & \cA(\R \times \R^n,\R^n) \ar@{=}[u]_{\cong}  \ar[ur]&
  }
\end{split}
\end{align}

In order to check that composition on $\DiffA$ is $C^\infty$ 
let $t \mapsto \Id + f(t,~)$ and $t \mapsto \Id + g(t,~)$ be in $\cD(\R,\DiffA)$. 
Then $f,g \in \cA(\R \times \R^n,\R^n)$, by \eqref{diag:A}. Consider 
\begin{equation} \label{eq:comp}
  ((\Id+f(t,~))\o(\Id+g(t,~)))(x)=x+g(t,x)+f(t,x+g(t,x))
\end{equation}
and define the $\cB_{\ge1}$-mapping (by Proposition \ref{prop:incl2})
\begin{align} \label{eq:ps}
\begin{split}
  \ps : \R \times \R^n \to \R \times \R^n, 
  ~(t,x) \mapsto (t,x+g(t,x)).
\end{split}
\end{align}
Then $f \o \ps \in \cA(\R \times \R^n,\R^n)$ and so composition is $C^\infty$ in view of \eqref{diag:A} and \eqref{eq:comp}, 
by Propositions
\ref{ssec:Bcomp},
\ref{ssec:Scomp},  
\ref{ssec:Wcomp}, and
\ref{ssec:Dcomp}; $\ps - \Id_{\R \times \R^n}$ tends to $0$ at infinity in the case $\cA=\cS$.

In order to see that inversion is $C^\infty$ 
let $t \mapsto \Id + f(t,~)$ be in $\cD(\R,\DiffA)$, and let $g$ be given by 
$(\Id + f)^{-1} = \Id+g$.  
Then $f\in \cA(\R \times \R^n,\R^n)$, by \eqref{diag:A}. 
Consider 
\begin{equation} \label{eq:implicit}
 (\Id +f) \o (\Id + g) = \Id \quad\Longleftrightarrow\quad  g(t,x) + f(t,x+g(t,x)) = 0.
\end{equation}
and define the $\cB_{\ge1}$-mapping (by Proposition \ref{prop:incl2})
\begin{align} \label{eq:ph2}
\begin{split}
  \ph : \R \times \R^n \to \R \times \R^n,
  ~(t,x) \mapsto (t,x+f(t,x)).
\end{split}
\end{align}
Then $\ph \o \ps = \ps \o \ph = \Id_{\R \times \R^n}$, where 
$\ps$ is the mapping defined in \eqref{eq:ps}.
By assumption $\inf_{(t,x) \in \R \times \R^n}\det d\ph(t,x) >0$, and hence Proposition \ref{prop:inv} implies that 
$\ps$ is $B_{\ge 1}$ and satisfies $\inf_{(t,x) \in \R \times \R^n}\det d\ps(t,x) >0$. 
So we may conclude that $g$ is $\cB_{\ge1}$ and therefore $\cB$, in view of \eqref{eq:implicit}. 

If $t \mapsto \Id + f(t,~)$ is in $\cD(\R,\DiffS)$, there is a compact interval $[a,b]$ 
such that if $t \not \in [a,b]$, then $f(t,~) = 0$ and so $g(t,~)=0$, by \eqref{eq:implicit}. 
It follows that $\ps - \Id_{\R \times \R^n}$ tends to $0$ at infinity. 
So we may conclude that $g\in \cS(\R \times \R^n,\R^n)$ by applying Proposition \ref{ssec:Scomp} 
to $g = -f \o \ps$ (that is \eqref{eq:implicit}).

The cases $\cA \in \{\cD,\Sp\}$ are analogous.
It follows that inversion on $\DiffA$ is $C^\infty$.

\subsection{$\Diff \DrM(\R^n)$, $\Diff \SrLM(\R^n)$, $\Diff W^{\rM,p}(\R^n)$, $\Diff \BrM(\R^n)$ are $\CrM$ Lie groups}

The arguments of Subsection \ref{ssec:smooth} provide the proof of this statement, if $M=(M_k)$ is non-quasianalytic. 
In fact, in this case a mapping is $\CrM$ if and only if it preserves $\CrM$-curves, or equivalently,
if it maps $\DrM$-curves to $\CrM$-curves; in the curve lemma \cite[3.6]{KMRc} the $\CrM$-curve may be chosen 
with compact support. Furthermore we need to replace Proposition \ref{ssec:smooth} by the following proposition.

\begin{proposition*} 
	Let $M=(M_k)$ be a log-convex derivation closed weight sequence.		
  Let $F \in \BM_{\ge 1}(\R^n,\R^n)$ be a diffeomorphism of $\R^n$ satisfying $\inf_{x \in \R^n}\det dF(x) >0$.
  Then $F^{-1} \in \BM_{\ge 1}(\R^n,\R^n)$ and $\inf_{x \in \R^n}\det dF^{-1}(x) >0$.
\end{proposition*}

That $\inf_{x \in \R^n}\det dF^{-1}(x) >0$ was shown in the proof of Proposition \ref{ssec:smooth}. That 
$F^{-1} \in \BM_{\ge 1}(\R^n,\R^n)$ follows e.g.\ from \cite{Komatsu79}, \cite{Yamanaka89}, \cite{Koike96}, 
see also \cite{KrieglMichorRainer14a} and \cite{RainerSchindl14}.

\def\cprime{$'$}
\providecommand{\bysame}{\leavevmode\hbox to3em{\hrulefill}\thinspace}
\providecommand{\MR}{\relax\ifhmode\unskip\space\fi MR }
\providecommand{\MRhref}[2]{%
  \href{http://www.ams.org/mathscinet-getitem?mr=#1}{#2}
}
\providecommand{\href}[2]{#2}


\begin{thebibliography}{10}

\bibitem{Chevet69}
S.~Chevet, \emph{Sur certains produits tensoriels topologiques d'espaces de
  {B}anach}, Z. Wahrscheinlichkeitstheorie und Verw. Gebiete \textbf{11}
  (1969), 120--138. \MR{0415347 (54 \#3436)}

\bibitem{FaadiBruno1855}
C.~F. Fa\`a~di Bruno, \emph{Note {s}ur {u}ne {n}ouvelle {f}ormule {d}u {c}alcul
  {d}iff\'erentielle}, Quart. J. Math. \textbf{1} (1855), 359--360.

\bibitem{GelfandShilov68}
I.~M. Gel{\cprime}fand and G.~E. Shilov, \emph{Generalized functions. {V}ol. 2.
  {S}paces of fundamental and generalized functions}, Translated from the
  Russian by Morris D. Friedman, Amiel Feinstein and Christian P. Peltzer,
  Academic Press, New York, 1968. \MR{0230128 (37 \#5693)}

\bibitem{Grothendieck55}
A. Grothendieck, \emph{Produits {t}ensoriels {t}opologiques {e}t {e}spaces {n}ucl\'eaires},
Memoirs Amer. Math. Soc. 16, 1955.

\bibitem{Koike96}
M.~Koike, \emph{Inverse mapping theorem in the ultradifferentiable class},
  Proc. Japan Acad. Ser. A Math. Sci. \textbf{72} (1996), no.~8, 171--172.
  \MR{1420610 (97k:58022)}

\bibitem{Komatsu77}
H.~Komatsu, \emph{Ultradistributions. {II}. {T}he kernel theorem and
  ultradistributions with support in a submanifold}, J. Fac. Sci. Univ. Tokyo
  Sect. IA Math. \textbf{24} (1977), no.~3, 607--628. \MR{0477770 (57 \#17280)}

\bibitem{Komatsu79}
\bysame, \emph{The implicit function theorem for ultradifferentiable mappings},
  Proc. Japan Acad. Ser. A Math. Sci. \textbf{55} (1979), no.~3, 69--72.

\bibitem{KM97}
A.~Kriegl and P.~W. Michor, \emph{The convenient setting of global analysis},
  Mathematical Surveys and Monographs, vol.~53, American Mathematical Society,
  Providence, RI, 1997, \url{http://www.ams.org/online\_bks/surv53/}.

\bibitem{KMRc}
A.~Kriegl, P.~W. Michor, and A.~Rainer, \emph{The convenient setting for
  non-quasianalytic {D}enjoy--{C}arleman differentiable mappings}, J. Funct.
  Anal. \textbf{256} (2009), 3510--3544.

\bibitem{KMRq}
\bysame, \emph{The convenient setting for quasianalytic {D}enjoy--{C}arleman
  differentiable mappings}, J. Funct. Anal. \textbf{261} (2011), 1799--1834.

\bibitem{KMRu}
\bysame, \emph{The convenient setting for {D}enjoy--{C}arleman differentiable
  mappings of {B}eurling and {R}oumieu type}, Rev. Mat. Complut. \textbf{28}
  (2015), no.~3, 549--597. \MR{3379039}

\bibitem{KrieglMichorRainer14a}
\bysame, \emph{An exotic zoo of diffeomorphism groups on {$\mathbb{R}^n$}}, Ann.
  Global Anal. Geom. \textbf{47} (2015), no.~2, 179--222. \MR{3313140}

\bibitem{LiebLoss01}
E.~H. Lieb and M.~Loss, \emph{Analysis}, second ed., Graduate Studies in
  Mathematics, vol.~14, American Mathematical Society, Providence, RI, 2001.
  \MR{1817225 (2001i:00001)}

\bibitem{Michor80}
P.~W. Michor,
\newblock \emph{Manifolds {o}f {d}ifferentiable {m}appings},
\newblock Shiva Mathematics Series 3, Orpington, 1980.

\bibitem{Michor80II}
P.~W. Michor,
\newblock \emph{Manifolds {o}f {s}mooth {m}aps {I}{I}: {T}he {L}ie {g}roup {o}f
  {d}iffeomorphisms {o}f a {n}on {c}ompact {s}mooth {m}anifold},
\newblock Cahiers Topol. Geo. Diff., \textbf{21} (1980), 63--86.

\bibitem{MichorMumford13}
P.~W. Michor and D.~Mumford, \emph{A zoo of diffeomorphism groups on
  {$\mathbb{R}^n$}}, Ann. Global Anal. Geom. \textbf{44} (2013), no.~4,
  529--540. \MR{3132089}

\bibitem{RainerSchindl12}
A.~Rainer and G.~Schindl, \emph{Composition in ultradifferentiable classes},
  Studia Math. \textbf{224} (2014), no.~2, 97--131.

\bibitem{RainerSchindl14}
\bysame, \emph{Equivalence of stability properties for ultradifferentiable
  function classes}, Rev. R. Acad. Cienc. Exactas Fis. Nat. Ser. A Math.
  RACSAM. (2015).

\bibitem{Schwartz66}
L.~Schwartz, \emph{Th{\'e}orie des distributions}, Publications de l'Institut
  de Math{\'e}matique de l'Universit{\'e} de Strasbourg, No. IX-X. Nouvelle
  {\'e}dition, enti{\'e}rement corrig{\'e}e, refondue et augment{\'e}e,
  Hermann, Paris, 1966. \MR{0209834 (35 \#730)}

\bibitem{Thilliez08}
V.~Thilliez, \emph{On quasianalytic local rings}, Expo. Math. \textbf{26}
  (2008), no.~1, 1--23.

\bibitem{Treves67}
Fran{\c{c}}ois Tr{{\`e}}ves, \emph{Topological vector spaces, distributions and
  kernels}, Academic Press, New York, 1967. \MR{0225131 (37 \#726)}

\bibitem{Yamanaka89}
T.~Yamanaka, \emph{Inverse map theorem in the ultra-{$F$}-differentiable
  class}, Proc. Japan Acad. Ser. A Math. Sci. \textbf{65} (1989), no.~7,
  199--202.

\bibitem{Ziemer89}
W.~P. Ziemer, \emph{Weakly differentiable functions}, Graduate Texts in
  Mathematics, vol. 120, Springer-Verlag, New York, 1989, Sobolev spaces and
  functions of bounded variation. \MR{1014685 (91e:46046)}

\end{thebibliography}

\end{document}